\documentclass[runningheads]{llncs}
\usepackage[inline]{enumitem}

\usepackage{amsmath,stmaryrd,color,bm,upgreek,enumitem,turnstile, multicol,amssymb}

\newcommand{\rnk}[1]{|#1|}
\newcommand{\putaway}[1]{}

\newcommand{\peq}{\preccurlyeq}
\newcommand\varOne{{\bm z}}

\newcommand\varThree{{\bm z}}
\newcommand\VarOne{{\bm X}}
\newcommand\VarTwo{{\bm Y}}

\newcommand\ConOne{{\bm X}}

\newcommand\ti{{\rm TI}}
\newcommand\wf{{\rm WF}}
\newcommand\wo{{\rm WO}}

\newcommand\ProofP[2]{[\tre]^{#1}_{#2}}
\newcommand\ProofPX[3]{[\tre]^{#1}_{#2}}
\newcommand\ModelsP[2]{[\mo]^{#1}_{#2}}
\newcommand\ModelsPX[3]{[\mo]^{#1}_{#2}}

\newcommand\RecP[2]{{[\ite]^{#1}_{#2}}}
\newcommand\RecPX[3]{[\ite]^{ }_{#2}}
\newcommand\FixP[2]{{[\fx]^{#1}_{#2}}}
\newcommand\ClosedP[2]{{[\clsd]^{#1}_{#2}}}
\newcommand\FixPX[3]{[\fx]^{#1}_{#2}}
\newcommand\ClosedPX[3]{[\clsd]^{#1}_{#2}}

\newcommand{\cpar}{\rho}

\newcommand{\picaform}{{\Pi}^1_1/{\bm \Sigma}^0_2}

\newcommand\vecgamma{{\Gamma}}
\newcommand\vecdelta{{\Delta}}
\newcommand\piti{{\sf TI}_0}
\newcommand\OmegaRule{{\omega\mbox{-}{\rm Rule}}}
\newcommand\RulePred{{\rm Rule}}

\newcommand\taitc{\mbox{\sc Tait}}

\newcommand\pica{{\bm \Pi}^1_1\mbox{-}{\sf CA}_0}

\newcommand{\Prop}[3]{ [{#2} ]^{#1}_{#3}}

\newcommand{\tre}{{\sf P}}
\newcommand{\ite}{{\sf R}}
\newcommand{\fx}{{\sf I}}
\newcommand{\clsd}{{\sf C}}
\newcommand{\mo}{{\sf M}}
\newcommand{\pvar}{\sf X}

\newcommand\spc[2]{{\rm SPC}^\rho_{#1}({#2})}
\newcommand\PsAlt\blacklozenge
\newcommand\NcAlt\blacksquare

\newcommand\yespr[1]{}

\newcommand\sat{S}

\newcommand\model{\ensuremath{\mathfrak M}}

\def\eca{\ensuremath{{{\sf RCA}^*_0}}\xspace}
\def\aca{{{\sf ACA}_0}}
\def\rcaa{{{\sf RCA}^\ast_0}}
\def\atr{{{\sf ATR}_0}}

\def\compax{{\rm CA}}

\def\lb{\left\llbracket}
\def\rb{\right\rrbracket}

\newcommand{\PRFN}[3]{\omega_{#1}\mbox{-}{\rm RFN}^{#2}_{#3}}
\newcommand{\PCons}[2]{\omega_{#1}\mbox{-}{\rm Cons}^{#2}}
\newcommand{\PCONS}[3]{\omega_{#1}\mbox{-}{\rm CONS}^{#2}_#3}

\newcommand\Neg{{\sim}}

\def\lb{\left\llbracket}
\def\rb{\right\rrbracket}

\def\nc{{\Box}}

\def\seq{\succcurlyeq}
\def\<{\left\langle}
\def\>{\right >}

\usepackage{xspace}

\newcommand{\pra}{\ensuremath{{\mathsf{PRA}}}\xspace}
\newcommand{\rca}{\ensuremath{{\mathsf{RCA}_0}}\xspace}

\newcommand{\pa}{\ensuremath{{\mathsf{PA}}}\xspace}
\newcommand{\Robinson}{\ensuremath{{\mathsf{Q}}}\xspace}

\newcommand{\ea}{\ensuremath{{\mathsf{EA}}}\xspace}

\begin{document}
\title{The many faces of omega-logic\thanks{Partially funded by the FWO-FWF Lead Agency Grant G030620N. }}
%
%
\author{David Fern\'andez-Duque\inst{1}\orcidID{0000-0001-8604-4183}}
\authorrunning{D. Fern\'andez-Duque}
%
\institute{Department of Mathematics WE16\\
Ghent University\\
Krijgslaan 281, S8\\
9000 Ghent,  Belgium\\
09 264 49 12\\
\email{david.fernandezduque@ugent.be}}
\maketitle              
\begin{abstract}
We consider several formalizations in the language of second-order arithmetic of {\em ``The formula $\phi$ is a theorem of $\omega$-logic'',} including some which have been studied in the literature and a new variant defined via a least fixed point.
We analyze the provability of relations between these different formalizations in standard theories of reverse mathematics.
With this, we study the strength of various reflection principles arising from these notions of provability, surveying known results and establishing some new equivalences, including a characterization of $\pica$ in terms of our fixed-point formalization of $\omega$-logic.
\end{abstract}

\section{Introduction}

The $\omega$-rule is an infinitary deduction rule that has the form
\[
\dfrac{ \ \ \phi(\bar 0),\Gamma \ \ \phi(\bar 1),\Gamma \ \ \phi(\bar 2),\Gamma \ \ \hdots \ \ }{\forall x\phi(x),\Gamma},
\]
with one premise for each natural number.
Augmenting the Tait calculus with this infinitary rule gives rise to {\em $\omega$-logic,} which can be readily formalized within second order arithmetic, although the precise details of the formalization may vary.
In fact, there are at least four ways to model $\omega$-logic in this context. Informally, they are:
\begin{enumerate}[label=({\it\roman*})]

\item\label{ItTreeDer} There is a well-founded derivation tree formalizing an $\omega$-proof of $\phi$, in which case we will write $\ProofP{ }{}\phi$.

\item\label{ItLevelDer} There is a well-order $\Lambda$ such that $\phi$ belongs to the set of theorems of $\omega$-logic defined by transfinite recursion on $\Lambda$, in which case we will write $\RecP{ }{}\phi$.

\item\label{ItClosed} The formula $\phi$ belongs to all sets closed under the rules and axioms of $\omega$-logic, which we denote $\ClosedP{ }{}\phi$.

\item\label{ItFix} The formula $\phi$ belongs to the least set closed under the rules and axioms of $\omega$-logic, which we denote $\FixP{ }{}\phi$.

\end{enumerate}
Although we will discuss these in greater detail later, intuitively $\ProofP{ }{}\phi$ gives a `local' view of $\phi$ being a theorem of $\omega$-logic by considering (infinite) $\omega$-proofs of $\phi$, while $\FixP{ }{}\phi$ gives a more global perspective, describing the set of theorems of $\omega$-logic as a whole via an inductive definition. Meanwhile, $\RecP{}{}\phi$ describes the approximations from below to the fixed point used in $\FixP{}{}\phi$ while $\ClosedP{}{}\phi$ describes the approximations from above.
The most standard of these formalizations is $\ProofP{ }{}\phi$ (see e.g.~\cite{Arai1998,GirardProofTheory}), but formalizations using transfinite recursion are convenient for establishing proof-theoretic semantics for the polymodal provability logics ${\sf GLP}_\Lambda$ with ordinal modalities \cite{FernandezJoosten:2013:OmegaRuleInterpretationGLP}.
Closely related is the notion:
\begin{enumerate}[label=({\it v})]

\item The formula $\phi$ is true on every $\omega$-model, which we denote $\ModelsP{ }{}\phi$.

\end{enumerate}
While this does not look like a notion of `provability', it is equivalent to $\omega$-provability, in view of the Henkin-Orey completeness theorem \cite{Orey56}.

Over a strong enough formal theory, one can show that all of these notions of $\omega$-provability are equivalent. However, from the point of view of a weak theory, they may vary in strength. For $\pvar \in \{\tre,\ite,{\sf C},\mo,\fx\}$ and $A\subseteq \mathbb N$, let us write $[{\pvar}] \phi (\dot A )$ if $\phi$ is provable in the sense of $\pvar$ from the atomic diagram of $A$. 

Theorem \ref{TheoEquivATR} states the provable equivalences between the various formalizations of $\omega$-logic.
We show that $\ProofPX {} T {\bm A}\phi (\dot{\bm A})$, $\ModelsPX {} T {\bm A}\phi (\dot{\bm A})$, and $\ClosedPX {} T {\bm A}\phi (\dot{\bm A})$ are provably equivalent over $\eca$.
Moreover, for $\rho<\omega$, $\ProofPX \rho T {\bm A}\phi (\dot{\bm A})$, $\ModelsPX \rho T {\bm A}\phi (\dot{\bm A})$, and $\ClosedPX \rho T {\bm A}\phi (\dot{\bm A})$ are provably equivalent over $\aca$.
That these notions are also equivalent to $\RecPX {} T {\bm A}\phi (\dot{\bm A})$ is provable in $\atr$, and that they are equivalent to  $\FixPX {} T {\bm A}\phi (\dot{\bm A})$ is provable in $\pica$.
This is largely a synthesis of known results, including the above-mentioned Henkin-Orey completeness theorem, but some equivalences are new, particularly those involving $\ite$ and $\fx $.

Once $\omega$-logic has been formalized, one can proceed to define the corresponding reflection principles.
{\em Reflection principles} in formal arithmetic are statements of the form {\em ``If $\phi$ is a theorem of $T$, then $\phi$''} \cite{KreiselLevy:1968:ReflectionPrinciplesAndTheirUse}. Using notation from provability logic \cite{Boolos:1993:LogicOfProvability}, for a computably enumerable theory $T$ we may use $\nc_T\phi$ to denote a natural formalization of {\em ``$\phi$ is a theorem of $T$''.} Then, the above statement may be written succinctly as $\nc_T\phi\to\phi$. If $\phi$ is a sentence, this gives us an instance of {\em local reflection,} which can almost never be proven within $T$ itself. For example, setting $\phi\equiv\, {\tt 0 =1}$, we see that $\nc_T\phi\to\phi$ is equivalent to $\Neg\nc_T{\tt 0 = 1}$, which asserts the consistency of $T$ and hence is unprovable within $T$ itself (if $T$ satisfies the assumptions of G\"odels second incompleteness theorem). More generally, by L\"ob's theorem we have that $T\vdash \nc_T\phi\to\phi$ {\em only} if $\phi$ is already a theorem of $T$ \cite{Lob:1955:SolutionProblemHenkin}.

We can extend reflection to formulas $\phi(x)$ to obtain {\em uniform reflection principles,} given by the scheme
\[\forall x \big (\nc_T\phi(\dot x)\to \phi(x)\big),\]
where $\dot x$ indicates that $x$ should be substituted by its numeral. When $\phi$ ranges over all formulas, this scheme is denoted ${\rm RFN}[T]$.
Uniform reflection principles are particularly appealing because they sometimes give rise to familiar theories. If we use $\pra$ to denote {\em primitive recursive arithmetic,} Kreisel and L\'evy~\cite{KreiselLevy:1968:ReflectionPrinciplesAndTheirUse} proved that
\[{\sf PA}\equiv{\sf PRA}+{\rm RFN}[{\sf PRA}];\]
in fact, we may replace $\pra$ by the weaker {\em elementary arithmetic} ($\ea$), obtained by restricting the induction schema in Peano arithmetic to $\Delta^0_0$ formulas and adding an axiom asserting that the exponential function is total \cite{Beklemishev:1997:InductionRules}.

Similar principles can be defined for $\omega$-logic.
For $\pvar \in \{\tre,\ite,{\sf C},\mo,\fx\}$, we define a schema
\[\PRFN{{\pvar}} {}{}  \equiv \forall A \ \forall n \ \big ( [{\pvar}] \phi(\dot n,\dot A) \to \phi (n,A)\big ).\]
If $\Gamma$ is a set of formulas, $\PRFN {\pvar}{}\Gamma$ is the restriction of this scheme to $\phi\in \Gamma$. Over $\rca$ we have that:
\begin{align}
\label{ItIntroOne} \PRFN \tre {}{}  &\equiv {\bm \Pi}^1_\omega\text{-}\piti;\\
\label{ItIntroTwo}  \PRFN \ite {}{{{\bm \Pi}^1_2}}&\equiv \atr.
\end{align}
(We will review all relevant theories of second order arithmetic in \S\ref{SecSOA}). The first item is proven in \cite{Arai1998} and the second in \cite{Cordon2017}. As we will see, if we use $\PRFN {\pvar}{}\Gamma[T]$ to denote a variant of the scheme where $\omega$-logic is extended by theorems of $T$, \eqref{ItIntroTwo} generalizes to
\begin{equation}\label{EqNewTheoremATR}
\PRFN \ite { }{{{\bm \Sigma}^1_{n+1}}}[\aca]\equiv \atr+{\bm \Pi}^1_n\text{-}{\rm TI}
\end{equation}
(which is just ${\bm \Pi}^1_n$-$\piti$ if $n>1$). Moreover, \eqref{ItIntroOne} also holds for $\omega$-model reflection, the scheme asserting that any formula true in every $\omega$-model must be true \cite{Jager1999}. This begs the question: is $\PRFN\fx {}{}$ also equivalent to a natural theory? In this article, we answer the question affirmatively, and prove that:
\begin{align}
\label{EqTheoOne}\PRFN\fx {}{{{\bm \Pi}^1_3}}&\equiv\pica;\\
\label{EqTheoTwo}\PRFN\fx {}{{{\bm \Sigma}^1_{n+1}}}[\aca]&\equiv\pica+{\bm \Pi}^1_n\text{\rm -TI}.
\end{align}
Both equivalences are proven over the theory $\eca$ of weak recursive comprehension, which is strictly weaker than $\rca$.

\subsection*{Layout of the article}

In \S \ref{SecSOA} we establish some basic notation we will use and review the subsystems of second-order arithmetic that will be of interest to us. In \S \ref{SecForm} we review formalizations of $\omega$-logic in the literature, and in \S \ref{SecMod} we review $\omega$-models and their corresponding reflection principles. In \S \ref{SecFix} we give our formalizations related to inductive definitions. In \S \ref{SecComp} we discuss completeness results for $\omega$-logic, with which we prove Theorem \ref{TheoEquivATR} and \eqref{EqNewTheoremATR}, and in \S \ref{section:OracleConsistencyAndOracleReflection} introduce the reflection principles based on our fixed point construction and prove partial results leading to \eqref{EqTheoOne} and \eqref{EqTheoTwo}. The latter are proven in \S \ref{section:CountableCodedModels} using $\beta$-models.

\section{Second-order arithmetical theories}\label{SecSOA}

In this section we review some basic notions of second-order arithmetic and mention some important theories that will appear throughout the article.
 
\subsection{Conventions of syntax}\label{SubSecSynt}

It will be convenient to work within a Tait-style calculus, so we will consider a language without negation, except on primitive predicates. Thus terms and formulas will be built from
the symbols ${\tt 0},{\tt 1}, {+} ,{\times},  = , \not=, \in,\not\in $ representing the standard constants, operations and relations on the natural numbers, along with the Booleans $\wedge,\vee$ and the quantifires $\forall,\exists$. If $\phi$ is any formula, its {\em rank,} $\rnk\phi$, is the number of logical symbols (Booleans and quantifiers) that appear in it. We assume a countably infinite set of first-order variables $n,m,x,y,z\dots$, as well as a countably infinite set of second-order variables. It will be convenient to assume that the second-order variables are enumerated by ${\bm V}=\langle V_i\rangle_{i\in \mathbb N}$, although we may also use $X,Y,Z,\dots$ to denote set-variables.  Tuples of first-order terms or second-order variables will be denoted with a boldface font, e.g. $\bm t$, $\bm X$. In general, if ${\bm S}= \langle S_i \rangle_{i\in\mathbb N}$ is a sequence and $n\in \mathbb N$ we will write ${\bm S}_{<n}$ for $\langle S_i\rangle_{i<n}$. We also include countably many set-constants ${\bm C}=\langle C_i\rangle_{i\in \mathbb N}$, which will be used to name `externally given' sets (see \S\ref{SubsOr}). 

We define $x\leq y$ by $\exists z \ (y=x+z)$ and $x<y$ by $x+{\tt 1}\leq y$. In the meta-language we may also use the symbol `$=$', although sometimes we use `$\equiv$' instead in order to distinguish it from the object-language equality. Since we have no negation in the language, we define $\Neg\phi$ by using De Morgan's laws and the classical dualities for quantifiers. In particular, we define $\phi\to\psi$ by $\Neg\phi\vee\psi$. The set of all formulas will be denoted ${\bm \Pi}^1_\omega$.

Fix some elementary G\"odel numbering mapping a formula $\psi\in{\bm \Pi}^1_\omega$ to a natural number $\ulcorner \psi \urcorner$ in such a way that terms and sequents of formulas are also assigned G\"odel numbers. Since we will be working mainly inside theories of arithmetic, we will often identify $\psi$ with $\ulcorner \psi \urcorner$. For a natural number $n$, define a term $\bar n$ recursively by $\bar 0={\tt 0}$ and $\overline{n+1}=(\bar n)+{\tt 1}$. We will assume that the G\"odel numbering has the natural property that $\ulcorner \psi\urcorner < \ulcorner \phi\urcorner$ whenever $\psi$ is a proper subformula of $\phi$.

We use ${\bm\Delta}^0_0$ to denote the set of all formulas, possibly with set parameters but without the occurrence of the set-constants $C_i$, where no second-order quantifiers appear and all first-order quantifiers are {\em bounded}, that is, of the form $\forall x< t \ \phi$ or $\exists x< t \ \phi$.
We simultaneously define ${\bm\Sigma}^0_{0}={\bm\Pi}^0_{0}={\bm\Delta}^0_0$ and recursively define ${\bm\Sigma}^0_{n+1}$ to be the set of all formulas of the form $\exists x \phi$ with $\phi\in {\bm\Pi}^0_{n}$, and similarly ${\bm\Pi}^0_{n+1}$ to be the set of all formulas of the form $\forall x \phi$ with $\phi\in {\bm\Sigma}^0_{n}$. We denote by ${\bm\Pi}^0_\omega$ the union of all ${\bm\Pi}^0_n$; these are the {\em arithmetical formulas.} 

The classes ${\bm\Sigma}^1_n,{\bm\Pi}^1_n$ are defined analogously, but using second-order quantifiers, and setting ${\bm \Sigma}_0^1 = {\bm\Pi}^1_0 = {\bm\Delta}^1_0 = {\bm\Pi}^0_\omega$. It is well-known that every second-order formula is equivalent to another in one of the above forms.  We use a lightface font for the analogous classes where no set-variables appear free: $\Delta^m_n,\Pi^m_n,\Sigma^m_n$. For lightface classes of formulas, we may write $\Gamma({\VarTwo})$ to indicate that the second-order variables in ${\VarTwo}$ may appear free (and no others). Finally, if $\Gamma$ is a set of formulas and $n$ is a natural number, we use ${\Pi}^1_n/\Gamma$ to denote the set of formulas of the form $\forall X_{n}\exists X_{n-1},\dots,Q_1 X_1 \phi$, with $\phi\in\Gamma$ and $Q_1\in \{\forall,\exists\}$.


We will also use {\em pseudo-terms} to simplify notation, where an expression $\phi(t({\bm x}))$ should be understood as a shorthand for $\exists y  \ \big (\psi({\bm x},y)\wedge\phi(y) \big )$, with $\psi$ a $\Delta^0_0$ formula defining the graph of the intended interpretation of $t$.
The domain of the functions defined by these pseudo-terms may be a proper subset of $\mathbb N$.

Let us list some of the (pseudo-)terms we will use:

\begin{enumerate}

\item A pseudo-term $2^x$ for the exponential function.

\item A term $\langle x,y\rangle$ which returns a code of the ordered pair formed by $x$ and $y$ and projection terms so that $(\langle x,y\rangle)_0 = x$ and $(\langle x,y\rangle)_1 = y$. We will overload this notation by also using it for sequences, coded in a standard way. As with tuples of variables, we use a boldface font when a first-order object is meant to be regarded as a sequence. For a sequence $\bm s$, we will also use $({\bm s})_i$ to denote a pseudo-term which picks out the $i^{\rm th}$ element of $\bm s$ if it exists, and is undefined otherwise, and $|{\bm s}|$ denotes a pseudo-term for the length of $\bm s$. If $n\in\mathbb N$, $\bm s * n$ denotes the sequence obtained by adjoining $n$ to $\bm s$ as its last element.
We will assume that for all $\bm s$ and $i<|\bm s|$ we have that $\max \{(\bm s)_i,|\bm s|\} \leq \bm s$.

\item A term $\overline x$ mapping a natural number to the code of its numeral.

\item A term $\lb x\rb$ which, when $x$ codes a closed term $t$, returns the value of $t$ as a natural number.

\item For every formula $\phi$ and variables $x_0, \ldots, x_m$, a term\linebreak $\phi(\dot x_0, \ldots, \dot x_m)$ which, given natural numbers $n_0, \ldots, n_m$, returns the code of the outcome of $\phi[{\bm x}/\bar{\varOne}]$, i.e., the code of $\phi(\bar n_0, \ldots, \bar n_m)$. We will often write such a term as $\phi(\dot{\bm x})$.
\end{enumerate}

Note that we may also use this notation in the meta-language. As is standard, we may define $X\subseteq Y$ by $\forall x (x\in X\rightarrow x\in Y)$, and $X = Y$ by $X\subseteq Y\wedge Y\subseteq X$. If the set $F$ is meant to represent a function, we may write $y=F(x)$ instead of $\langle x,y\rangle\in F$. {\em Sequents} will be first-order objects of the form $\vecgamma=\langle\gamma_1,\dots,\gamma_n\rangle$, where each $\gamma_i$ is a formula. We will treat sequents as sets, defining $\phi\in \vecgamma$ by $\exists i< |\vecgamma| \ \phi=(\vecgamma)_i$, and define $\vecgamma\subseteq\vecdelta$ similarly.
We may write $\vecgamma,\phi$ or $(\vecgamma,\phi)$ instead of $\vecgamma\ast\phi$. We similarly use $\vecgamma,\vecdelta$ to denote the concatenation of $\vecgamma$ and $\vecdelta$. The empty sequent will be denoted by $\bot$; observe that we do not take it to be a symbol of our formal language.

\subsection{Basic rules and axioms}

We will work with a one-sided Tait-style calculus, which proves sequents of the form $\vecgamma=\langle\gamma_i\rangle_{i<n}$, as defined in e.g. \cite{Pohlers:2009:PTBook}. In such a calculus, negation may only be applied to atomic formulas.
We assume that the Tait calculus contains enough axioms and rules so that $\vecgamma,\alpha$ is derivable whenever $\alpha$ is a true atomic sentence; this is not a strong assumption, as standard calculi have this property.
We will also assume that at least the following rules are available:
\[
\begin{array}{cclcc}
{(\mbox{\sc lem})}&\displaystyle\frac{}{\vecgamma,\alpha,\Neg\alpha}&&  
{({=})}&\displaystyle\frac{\vecgamma, \ \alpha \phantom{blabla} \vecgamma, r=r'}{\vecgamma, \ \alpha'}
\\\\
(\wedge)&\displaystyle\frac{\vecgamma,\phi \phantom{blabla} \vecgamma,\psi}{\vecgamma,\phi\wedge\psi}&
&(\vee) &\displaystyle\frac{\vecgamma,\phi,\psi}{\vecgamma,\phi\vee\psi}  \\\\
(\forall^0) & \displaystyle\frac{\vecgamma,\phi(v)}{\vecgamma,\forall x\phi(x)}& \phantom{(v not free )}&
(\exists^0)&\displaystyle\frac{\vecgamma,\phi(t)}{\vecgamma,\exists x\phi(x)} \\\\
(\forall^1) & \displaystyle\frac{\vecgamma,\phi(V)}{\vecgamma,\forall X\phi(X)}& \phantom{(v not free )}&
(\exists^1)&\displaystyle\frac{\vecgamma,\phi(Y)}{\vecgamma,\exists X\phi(X)} \\\\
(\mbox{\sc cut})&\displaystyle\frac{\vecgamma,\phi\phantom{blabla}\vecgamma,\Neg\phi}\vecgamma,&&
\end{array}
\]
where $\alpha$ is atomic, $v,V$ do not appear free in $\vecgamma$, and $\alpha'$ is obtained from $\alpha$ by replacing some instances of $r$ by $r'$.
We denote this calculus by $\taitc$; $\taitc^\rho$ is the restriction of $\taitc$ which allows cuts only for formulas of rank less than $\rho\leq\omega$ (in particular, $\taitc=\taitc^\omega$).

We identify a theory $T$ with its set of axioms, and a sequent $\Gamma$ is derivable in $T$ (denoted $T\vdash \Gamma$) if it is derivable in the Tait calculus augmented with the rule
\[(\text{\sc Ax}_T) \ \ \ \ \  \displaystyle\frac{\phantom{aaaaaa}}{ \alpha }, \]
for $\alpha$ an axiom of $T$.
Given an axiomatically presented $T$, $\taitc_T$ is the Tait calculus enriched with this rule.
By convention we always allow cuts to be applied to axioms, and $\taitc^\rho_T$ is the restriction where cuts are applied either to axioms of $T$ {\em or} to formulas of rank less than $\rho$.

\subsection{Successor induction and comprehension}\label{SubsecSSOA}

As our `background theory' we will use Robinson's arithmetic $\sf Q$ \cite{HajekPudlak:1993:Metamathematics} (essentially, $\pa$ without induction).
Aside from the basic axioms of $\Robinson$, the following schemes will be useful in axiomatizing many theories of interest to us. Below, $\Gamma$ denotes a set of formulas.

\begin{description}

\item[$\Gamma\mbox{-}\compax$] $\exists X\forall x\ \big (x\in X\leftrightarrow \phi(x)\big )$, where $\phi\in\Gamma$ and $X$ is not free in $\phi$;

\item[${\bm \Delta}^0_1\mbox{-}\compax$] $\forall x \big (\pi(x)\leftrightarrow\sigma(x) \big )\rightarrow\exists X\forall x\ \big (x\in X\leftrightarrow \sigma(x)\big )$,
where $\sigma\in{\bm \Sigma}^0_1$, $\pi\in{\bm \Pi}^0_1$, and $X$ is not free in $\sigma$ or $\pi$;

\item[${\rm I}\Gamma$] $\phi({\tt 0})\wedge\forall x\, \big (\phi(x) \to\phi(x+ {\tt 1}) \big )\ \to\ \forall x \ \phi(x)$, where $\phi\in\Gamma$;

\item[${\rm Ind}$] ${\tt 0}\in X\wedge \forall x\ \big (x\in X\rightarrow x+{\tt 1}\in X \big )\ \to\ \forall x\, (x\in X).$

\end{description}
In addition, $\rm Exp$ is a formula stating that the exponential function is total.
With this, we may define the following theories:
\begin{center}
\begin{tabular}{rll}
$\eca$&$\equiv$ &$\Robinson + {\rm Exp} + {\rm Ind}+{\bm\Delta}^0_1$-$\compax$;\\
$\rca$&$\equiv$ &$\Robinson + {\rm I}{\bm\Sigma}^0_1+{\bm\Delta}^0_1$-$\compax$;\\
$\aca $&$\equiv$ &$\Robinson + {\rm Ind}+{\bm\Sigma}^0_1$-$\compax$;\\
$\pica $&$\equiv$ &$\Robinson + {\rm Ind}+{\bm\Pi}^1_1$-$\compax$.\\
\end{tabular}
\end{center}
We use a Roman font for axioms or schemes and a sans-serif font for theories.
Later we will make use of the fact that (in particular) $\aca$ is finitely axiomatizable \cite[Lemma
VIII.1.5]{Simpson:2009:SubsystemsOfSecondOrderArithmetic}.
We assume that all pseudo-terms are defined so that $\eca$ proves that they are total functions on their intended domain.

Next, it will be useful to give a somewhat more economical (but equivalent) representation of $\pica$.

\begin{theorem}\label{TheoPicaSimp}
The theory $\pica$ is equivalent to $\Robinson + {\rm Ind}+({\Pi}^1_1/{\bm \Sigma}^0_2)\mbox{-}\compax$.
\end{theorem}

\begin{proof}[sketch]
In \cite[Lemma V.1.4]{Simpson:2009:SubsystemsOfSecondOrderArithmetic}, it is proven that any ${\bm \Pi}^1_1$ formula is equivalent to one of the form $\forall f \colon\mathbb N\to\mathbb N \ \phi (f),$ where $\phi\in {\bm \Sigma} ^0_1$. If ${\rm fun}(F)\in{\Pi}^0_2(F)$ is a formula stating that $F$ is the graph of a function, this is in turn equivalent to some formula $\forall F \ \big ( {\rm fun}(F) \to \phi'(F) \big )\in {\Pi}^1_1/{\bm \Sigma}^0_2,$ where $\phi'$ is obtained from $\phi$ by replacing each instance of $f(t)$ by a pseudo-term for the unique $y$ such that $\langle t,y\rangle \in F$.
\end{proof}

\subsection{Transfinite recursion and bar induction}

We mention two further theories that will appear later and require a more elaborate setup. We may represent well-orders in second-order arithmetic as pairs of sets $\Lambda=\langle |\Lambda|,<_\Lambda\rangle$, and write $\lambda < \Lambda$ instead of $\lambda \in |\Lambda|$. Then, we define
\begin{align*}
{\rm Prog}_\phi (\Lambda)&=\forall \lambda \ \Big( \big ( \forall \xi {<_\Lambda} \lambda \ \phi(\xi) \big) \to \phi(\lambda) \Big )\\
\ti_\phi(\Lambda)&=\forall \lambda {<} \Lambda \ \big (
{\rm Prog}_\phi(\Lambda)
 \rightarrow \forall \lambda  {<} \Lambda \ \phi(\lambda)\big) \phantom{\Big(}\\
\wf(\Lambda)&=\forall X \ \ \ti_{\lambda\in X}(\Lambda) \phantom{\Big(}\\
\wo(\Lambda)&={\rm LO}(\Lambda)\wedge \wf(\Lambda), \phantom{\Big(}
\end{align*}
where ${\rm LO}(\Lambda)$ is a formula expressing that $\Lambda$ is a linear order.

Given a set $X$ whose elements we will regard as ordered pairs $\langle\lambda,n\rangle$, let $X_{\lambda}$ be the set of all $n$ with $\langle \lambda,n\rangle\in X$, and $X_{<_\Lambda \lambda}$ be the set of all $\langle \eta,n\rangle$ with $\eta<_\Lambda \lambda$. With this, we define the {\em transfinite recursion} scheme by
\[{\rm TR}_\phi(X,\Lambda)= \forall \lambda {<} \Lambda \ \forall n \ \big (n\in X_\lambda\leftrightarrow \phi(n,X_{<_\Lambda\lambda}) \big ).\]
For a set of formulas $\Gamma$ we define the schemes
\begin{center}
\begin{tabular}{rll}
$\Gamma$-$\rm TR$&$\equiv$&$\Big \{ \forall \Lambda \big ( {\rm WO}(\Lambda)\rightarrow \exists X \ {\rm TR}_\phi(X,\Lambda) \big ): \phi\in \Gamma \Big\}$;\\
$\Gamma\text{-}{\rm TI} $&$\equiv$&$\Big \{ \forall \Lambda \big ( {\rm WO}(\Lambda)\rightarrow {\rm TI}_\phi(\Lambda) \big ) :\phi\in \Gamma \Big\}$,
\end{tabular}
\end{center}
and the theories $\atr = \rca + {\bm \Pi} ^0_\omega$-$\rm TR$ and $\Gamma\text{-}\piti = \rca + \Gamma\text{-}{\rm TI} $.

Often $\atr$ and $\Gamma\text{-}\piti $ are defined over $\aca$ rather than $\rca$, but the two definitions are equivalent.
We will make use of the following relations between theories.
\begin{proposition}\label{propRelBetTheos}\
\begin{enumerate}

\item $\aca \equiv {\bm \Sigma}^0_1\text{-}\piti \subseteq \atr $;

\item ${\bm \Pi}^1_1\text{-}\piti\not\subseteq \atr \subsetneq{\bm \Sigma}^1_1\text{-}\piti$.

\end{enumerate}

\end{proposition}

The left-hand of the first item is shown in \cite{hirst_1999} and the right is straightforward, as arithmetical comprehension follows from transfinite recursion applied to the well-order $1$.
The second is proven in \cite{Simpson82}; to be precise, ${\bm \Pi}^1_1\text{-}\piti\equiv{\bm \Sigma}^1_1$-${\sf DC}_0$, a theory known to be incomparable with $\atr$.

\section{Formalized $\omega$-logic}\label{SecForm}

In this section we will give the necessary definitions in order to reason about $\omega$-logic within second-order arithmetic, and introduce the provability operator $\ProofP{}{}$ based on $\omega$-proofs.

\subsection{Formalized deduction}\label{SecFormal}

For our purposes, a (formalized) {\em theory} $T$ is a set of formulas defined by a ${\bm \Delta}^0_0$ formula ${\rm Ax}_T(x)$, representing the axioms of $T$.
For $\cpar\leq\omega$, fix $\RulePred^\cpar(x,y)\in {\Delta}^0_0 $ such that it is provable in $\eca$ that if $\RulePred^\cpar(x,y)$ holds, then $x$ codes a sequence of sequents $\langle \vecdelta_i\rangle_{i < n}$ and $y$ codes a sequent $\vecgamma$, and such that $\displaystyle\frac{\langle \vecdelta_i\rangle_{i < n}}{\vecgamma}$ is an instance of a rule of $\taitc^\cpar$ if and only if $\RulePred^\cpar(\langle \vecdelta_i\rangle_{i < n},\vecgamma)$ holds.
Similarly, given a theory $T$, $\RulePred^\cpar_T(x,y)$ expresses that $x,y$ represent an instance of a rule of $\taitc^\cpar_T$; recall that cuts are always allowed to be applied to axioms of $T$. 
We assume that the Tait calculus is formalized in such a way that the scheme stating that $\vecgamma,\alpha$ is derivable whenever $\alpha$ is a true atomic sentence is provable in $\eca$; note that $\Sigma^0_1$-completeness is provable in $\ea$ for standard calculi \cite{HajekPudlak:1993:Metamathematics}.

We also need to formalize the infinitary Tait calculus with the $\omega$-rule, which we denote by $\omega$-$\taitc$. Recall that this rule has infinitely many premises, and the following form:
\[\dfrac{\langle\vecgamma,\phi(\bar n):n\in\omega\rangle}{\vecgamma,\forall x\ \phi(x)}.\]
We can formalize this using the following expression:
\begin{align*}
\OmegaRule(P,\vecgamma)&\equiv \exists \phi\in\vecgamma \ \exists x , \psi < \phi \ \Big ( \phi=\forall x \psi(x) \wedge \forall z \big (\vecgamma,\psi(\dot z)\in P \big ) \Big ). \\
\end{align*}
Here, $P$ is a set-variable and $\psi < \phi$ refers to the standard order on natural numbers (recalling that we identify formulas with their G\"odel codes). The formula $\OmegaRule(P,\vecgamma)$ states that $\vecgamma$ follows by applying one $\omega$-rule to elements of $P$, and will be used in our formalizations of $\omega$-logic.

\subsection{Theories with oracles}\label{SubsOr}

In order to deal with free second-order variables, we will enrich theories with oracles. As we have mentioned previously, we will use countably many constants ${\bm C}=\langle C_i\rangle_{i\in \mathbb N}$ in order to `feed' information about any tuple of sets of numbers into $T$. The $C_i$'s are assumed to be disjoint from the second-order variables.

To be precise, we first encode finite sequences of sets in a natural way: for example, we may encode $\langle A_i\rangle_{i<n}$ by
\[{\bm A}= \big \{\langle 0,n\rangle \big \} \cup \big \{ \langle k,i+1\rangle :k\in A_i \wedge i<n\big \}.\]
The pair $\langle 0,n\rangle$ is included in order to know the length of the sequence, in case that e.g. $A_{n-1}=\varnothing$. As with tuples of natural numbers, we write $n=|{\bm A}|$.

Then, given a theory $T$ and a set-tuple ${\bm A}$, define $T|{\bm A}$ to be the extension of $T$ with the new axioms
\begin{center}
\begin{tabular}{llr}
($\mbox{\sc o}_\in$)&$\displaystyle\frac{\phantom{blablablablab}}{\vecgamma,\overline k\in C_i}$ &\ for $k\in A_i$ and $i<|{\bm A}|$ \\\\
($\mbox{\sc o}_{\not \in}$)&$\displaystyle\frac{\phantom{blablablablab}}{\vecgamma,\overline k\not \in C_i}$ &\  for $k\not \in A_i$ and $i<|{\bm A}|$.
\end{tabular}
\end{center}
It should be clear that given a formalized axiomatization ${\rm Ax}_T $ for $T$ we can define a new ${\bm \Delta}^0_0$ axiomatization ${\rm Ax}_{T|{\bm A}} $ for $T|{\bm A}$.

\subsection{Formalizing proof trees}

In \cite{Arai1998,GirardProofTheory}, derivability in $\omega$-logic is formalized by the existence of (typically infinite) derivation trees.
These are represented in second-order arithmetic using $\omega$-trees.

\begin{definition}\label{deOmTrees}
Let $\mathbb N^{<\omega}$ denote the set of all finite sequences of natural numbers.
If ${\bm s},\bm t\in \mathbb N^{<\omega}$, define $\bm s \peq \bm t$ if $\bm s$ is an initial segment of $\bm t$, and $\mathop\downarrow \bm s=\{\bm t\in S:\bm t\peq \bm s\}$. An {\em $\omega$-tree} is a set $S\subseteq \mathbb N^{<\omega}$ such that $\mathop\downarrow S=S$. A {\em labelled $\omega$-tree} is a pair $\langle S,L\rangle $ such that $S$ is an $\omega$-tree and $L\colon S\to\mathbb N$.
We say that $S$ is well-founded if $\seq$ is well-founded, i.e.~if there are no infinite strictly {\em increasing} sequences.
\end{definition}

\begin{definition}
A {\em preproof (for $T$) of cut-rank at most $\rho\leq\omega$} is a labeled $\omega$-tree $\langle S,L\rangle$ such that for every $\bm s\in S$, $L(\bm s)$ is a sequent, and  there is an instance $\displaystyle\frac{\langle \vecdelta_i\rangle_{i < \xi }}{\vecgamma}$ of a rule of $\omega$-$\taitc^{\rho}_T$ with $\xi\leq \omega$ such that $L(\bm s)=\vecgamma$ and for all $i\in\mathbb N$, $\bm s\ast i\in S$ if and only if $i<\xi$, in which case $L(\bm s\ast i)=\vecdelta_i$.

Let ${\rm Preproof}^\rho_T(S,L)$ be a $\Pi^0_1(S,L)$ formula stating that $\langle S,L\rangle$ is a preproof for $T$ of cut-rank at most $\rho$.
If $S$ is (upwards) well-founded, we will say that $\langle S,L\rangle$ is an {\em $\omega$-proof}.
If $S$ is finite, we say that $\langle S,L\rangle$ is a {\em finitary proof}.
\end{definition}

The formula ${\rm Preproof}^\rho_T(S,L)$ would make use of the formulas $\RulePred$ and (a mild variant of) $\OmegaRule$ defined in \S \ref{SecFormal}; this is developed in much more detail, for example, in \cite{GirardProofTheory}.

\begin{definition}
Given $\rho\leq\omega$, define a formula $\ProofP \rho T\vecgamma$ by
\[\exists S \ \exists L \ \Big( {\rm WF}(\langle S,{\succcurlyeq}\rangle )\wedge {\rm Preproof}^\rho_T(S,L)\wedge L(\langle\rangle)=\vecgamma \Big).\]
We write $\ProofPX \rho T{\bm X}\vecgamma (\dot {\bm X})$ instead of $\ProofP \rho {T|{\bm X}}\vecgamma$.
\end{definition}

We may omit the parameter $\rho$ and write $\ProofP {} {T|{\bm X}}\vecgamma$ when $\rho=\omega$.
For the sake of uniformity, we will also use proof trees to formalize deduction in standard theories.
Let $\vecgamma$ be a first order variable, and $\nc_T \vecgamma$ be a ${\bm \Sigma}^0_1$ formula stating that there is a finitary proof of $\vecgamma$, coded as a single natural number.
We call $\nc_T \vecgamma$ the {\em provability predicate for $T$.}

The following is immediate from the definition:

\begin{lemma}\label{LemmPMonotone}
Given $\rho\leq\sigma\leq\omega$, it is provable in $\eca$ that for all tuples of sets $\bm A$, $\ProofPX \rho TX\vecgamma (\dot{\bm A})$ implies $\ProofPX \sigma TX\vecgamma(\dot{\bm A})$.
\end{lemma}

The notion of provability $\ProofP {}{}$ gives rise to a natural reflection scheme.

\begin{definition} \label{DefRFNProof}
Given a theory $T$, $\rho\leq\omega$, and a set of formulas $\Gamma$, we define a schema
\begin{align*}
\PRFN\tre\rho  {\Gamma}[T] &\equiv\forall {{\bm X}} \, \forall {{\bm z}}\ \Big(\, \ProofPX \rho T {{\bm X}} \, \phi(\dot{{\bm z}},{\dot{\bm X}}) \to \phi({{\bm z}},{{\bm X}}) \Big ),
\end{align*}
where $\phi({\bm z},\bm X)\in \Gamma$ with all free variables among those shown.
\end{definition}

We may omit the parameter $\rho$ when $\rho=\omega$, as well as the parameter $T$ when $T$ is just the Tait calculus. This form of reflection gives an alternative axiomatization for transfinite induction, as shown by Arai \cite{Arai1998}.

\begin{theorem}\label{TheoArai}
$\rca + \PRFN\tre {}{{{\bm \Pi}^1_\omega}} \equiv  {{\bm \Pi}^1_\omega}\text{-}\piti.$
\end{theorem}

Note the analogy with Kreisel and L\'evy's result; just as reflection is equivalent to induction, $\omega$-reflection is equivalent to transfinite induction. As we will see, different formulations of $\omega$-logic also give rise to certain forms of comprehension.

\section{Countable $\omega$-models and reflection}\label{SecMod}

Another notion of reflection can be defined using $\omega$-models. An $\omega$-model is a second-order model whose first-order part consists of the standard natural numbers with the usual arithmetical operations. Because this part of our model is fixed, we only need to specify the second-order part, which consists of a family of sets over which we interpret second-order quantifiers. Moreover, if this family is countable, we can represent it using a {\em single} set.
If $\bm M$ codes a sequence of sets, a satisfaction class on $\bm M$ is a set which obeys the usual recursive clauses of Tarski's truth definition, where each constant $C_n$ is interpreted as $ M_ n$. Let us give a precise definition:

\begin{definition}
Let $\bm M\subseteq \mathbb N$. A
\emph{satisfaction class on $\bm M$} is a set $\sat\subseteq {\bm \Pi}^1_\omega(\bm C)$ such that, for any terms $t,s$, $n\in \mathbb N$, and sentences $\phi,\psi,$
\[
\begin{array}{rll}
(t\circ s)\in \sat & \Rightarrow & \lb t\rb\circ \lb s\rb \ \ (\circ\in \{=,\not=\});\\
(t\circ C_n) \in \sat &\Rightarrow &\langle n,\lb t\rb\rangle \circ \bm M \ \   (\circ\in \{\in,\not\in\}) ;\\
(\phi\wedge \psi) \in \sat & \Rightarrow & \phi\in \sat
\mbox{ and }
\psi \in \sat;\\
(\phi\vee \psi) \in \sat & \Rightarrow & \phi\in \sat
\mbox{ or }
\psi \in \sat;\\
(\exists u\,\phi(u)) \in \sat & \Rightarrow &  \mbox{for some }n\in
\mathbb{N},\ \phi(\bar n)\in \sat;\\
(\forall u\,\phi(u)) \in \sat & \Rightarrow &  \mbox{for all }n\in
\mathbb{N},\ \phi(\bar n)\in \sat;\\
(\exists X\,\phi(X)) \in \sat & \Rightarrow &  \mbox{for some }n\in
\mathbb{N},\ \phi(C_n)\in \sat;\\
(\forall X\,\phi(X)) \in \sat & \Rightarrow &  \mbox{for all }n\in
\mathbb{N},\ \phi(C_n)\in \sat.
\end{array}
\]
Given a set of sentences $\Gamma\subseteq {\Pi}^1_\omega(\bm C)$ closed under subformulas and substitution by closed terms (including set-constants), if for every $\phi\in \Gamma $ we have that either $\phi\in \sat$ or $\Neg\phi\in \sat$, we will say that $\sat$ is a {\em $\Gamma$-satisfaction class.} If $\Gamma$ contains all formulas of rank $\rho\leq \omega$, we say that $S$ is a {\em satisfaction class of rank $\rho$.} A pair $\model=\langle |\model|,S_\model\rangle$, where $|\model|$ is a sequence of sets and $S_\model$ is a $\Gamma$-satisfaction class on $|\model|$ of rank $\rho$ is a {\em $\Gamma$-valued $\omega$-model of rank $\rho$.} If $\Gamma$ is the set of all sentences of ${\Pi}^1_\omega(\bm C)$, we say that $\model$ is a {\em full $\omega$-model.}
\end{definition}

Satisfaction classes are used to define truth in a model:

\begin{definition}
Given an $\omega$--model $\model$, we write $\model\models\phi$ if $\phi\in S_\model$. If $T$ is a theory, we say that $\model$ is an $\omega$-model of $T$ if, whenever $\phi$ is an axiom of $T$, it follows that $\model \models \phi$.
If $\bm A$ is an $a$-tuple of sets, we write $\ModelsPX\rho T{\bm A}\phi (\dot{\bm A})$ for the formula stating that, for every $\Gamma$-valued $\omega$-model $\model$ of rank at least $\rho$ of $T$ with $\phi\in\Gamma$ and $|\model|_{<a}={\bm A}$, $\model\models\phi$.
\end{definition}

As with $[{\sf P}]$, we write $\ModelsPX{} T{\bm A}\phi (\dot{\bm A})$ when $\rho=\omega$, and will adhere to the same convention for other notions of $\omega$-provability.
Since the first-order part of an $\omega$-model is just the natural numbers, it is easy to see that, for arithmetical sentences, truth in a model is equivalent to truth. This partially extends to ${\bm \Pi}^1_1$-sentences:

\begin{lemma}\label{LemModSound}
Let $T$ be any theory and $\rho \leq\omega$. Then, if $\phi(\bm z,\bm X)\in {\bm \Pi}^1_1$ with all free variables shown,
\[\eca\vdash \forall {\bm A} \ \forall \bm n \ \big (\phi(\bm n,{\bm A}) \rightarrow \ModelsPX\rho T{\bm A} \phi(\dot{\bm n},\dot {\bm A})\big ).\]
\end{lemma}

\proof
First assume that $\phi$ is arithmetical, and let $\model$ be a model of $T$ of rank $\rho$. Then, an external induction using the definition of a satisfaction class shows that, if $\phi$ holds, then $\model\models\phi$. Otherwise, assume that $\phi=\forall X \ \psi(X)$ and $\model\not\models \forall X \psi (X) $, so that $\model \not \models \psi(C_k)$ for some $k$. But then, by the arithmetical case, $ \psi(C_k)$ fails, so that $\forall X \ \psi(X)$ fails.
\endproof

The first item of the following claim is immediate from observing that every model of rank $\sigma$ is already a model of any rank $\rho\leq\sigma$. The second is follows from \cite[Theorem VIII.1.13]{Simpson:2009:SubsystemsOfSecondOrderArithmetic}, which states that any set can be included in a full $\omega$-model of $\aca$.

\begin{lemma}\label{LemModPar}
Let $\phi (\bm z,\bm X)$ be an arbitrary formula with free variables among those shown and $\rho\leq\sigma\leq\omega$. Then,
\begin{enumerate}

\item \label{itModParECA} $\eca\vdash \forall {\bm X} \forall{\bm z} \ \big (\ModelsPX\rho T{\bm X} \phi(\dot{\bm z}, {\dot{\bm X}}) \rightarrow \ModelsPX\sigma T{\bm X}\phi( \dot{\bm z}, {\dot{\bm X}})\big );$

\item \label{itModParATR} $\atr \vdash \forall {\bm X} \forall{\bm z} \ \big (\ModelsPX\rho T{\bm X} \phi({\dot {\bm z}} , {\dot{\bm X}}) \leftrightarrow \ModelsPX\sigma T{\bm X}\phi(\dot{\bm z}, {\dot{\bm X}})\big ).$

\end{enumerate}
\end{lemma}

We may use $\omega$-models to define a notion of reflection $\PRFN\mo\rho {\Gamma}[T]$ analogously to Definition \ref{DefRFNProof}. The following is proven by J\"ager and Strahm \cite{Jager1999}, and is a refinement of results of Friedman \cite{Friedman75} and Simpson \cite{Simpson82}:

\begin{theorem}\label{TheoModRFN}
$\aca+ \PRFN\mo 0 {{\bm \Sigma}^1_{1+n}}[\aca]  \equiv {{\bm \Pi}^1_n}\text{-}\piti.$
\end{theorem}

\begin{remark}
In the literature, $\omega$-model reflection is often presented as {\em `If $\phi$ is true, then $\phi$ is satisfiable in an $\omega$-model'.} We have presented it dually as {\em `If $\phi$ holds in every $\omega$-model, then $\phi$ is true'.} The two schemes are clearly equivalent, but we prefer the latter for its symmetry with the other notions of reflection we consider. Note, however, that we must replace $\phi$ by ${\sim}\phi$ to pass from one to the other, and thus Theorem \ref{TheoModRFN} is stated with ${\bm \Sigma}^1_{n+1}$ in place of ${\bm \Pi}^1_{n+1}$ as in \cite{Jager1999}.
\end{remark}

\section{Inductive definitions of $\omega$-logic}\label{SecFix}

We may also formalize {\em `provable in $\omega$-logic'} in second-order arithmetic using a least fixed point construction. To this end, let us review how such fixed points may be treated in this framework.

\subsection{Inductive definitions}\label{SecIndDef}

Below, recall that we are working in a language without negation for non-atomic formulas.

\begin{definition}
Let $\phi$ be any formula and $X$ a set-variable. We say {\em $\phi$ is positive on $X$} if $\phi$ contains no occurrences of $t\not \in X$.
\end{definition}

A positive formula $\phi$ induces a map $F=F_\phi\colon 2^\mathbb N\to 2^\mathbb N$, which is monotone in the sense that $X\subseteq Y$ implies that $F (X)\subseteq F (Y)$. It is well-known that any such operator has a least fixed point.

\begin{definition}
Given a formula $\phi(n,X)$, we define the abbreviations
\begin{align*}
{\rm Closed}_\phi (X)&\equiv \forall n \ \big (\phi(n,X)\rightarrow n\in X \big )\\
\big (X=\mu X.\phi\big) &\equiv {\rm Closed}_\phi(X)\wedge \forall Y\big ({\rm Closed}_\phi(Y)\rightarrow X\subseteq Y\big ).
\end{align*}
\end{definition}

It is readily checked that $n\in \mu X.\phi $ if and only if $\phi(n,\mu X.\phi)$ holds. Such fixed points can be constructed `from below' using transfinite iterations of $F$: if we define $F^0(X)=X$, $F^{\xi+1}(X)=F (F^\xi(X))$ and $F^\xi(X)=\bigcup _{\zeta<\xi}F^\zeta(X)$, then by cardinality considerations one can see that
\begin{equation}\label{EqFromBelow}
\mu X. \phi=F^{\omega_1}(\varnothing).
\end{equation}
On the other hand, we may define $\mu X.\phi$ `from above' as the intersection of all sets $Y$ such that ${\rm Closed}(Y)$ holds. The latter definition is available in $\pica$, as is well-known (see e.g. \cite{IteratedInductive}), and thus we see that:

\begin{lemma}\label{LemmExistsFix}
Given $\phi(X)\in {\bm \Pi}^0_\omega$ which is positive on $X$, it is provable in $\pica$ that $\exists Y \ \big (Y=\mu X.\phi\big ).$
\end{lemma}

In particular, the rules of $\omega$-logic give rise to a positive operator, and a theorem of $\omega$-logic is any element of its least fixed point. Below, we develop this idea to give alternative formalizations of $\omega$-logic.

\subsection{The recursive formalization of $\omega$-logic}

We may use \eqref{EqFromBelow} to formalize `$\phi$ is a theorem of $\omega$-logic', as in \cite{Cordon2017,FernandezJoosten:2013:OmegaRuleInterpretationGLP}. There, provability along a countable well-order $\Lambda$ is modeled using an `iterated provability class' $P$, defined by arithmetical transfinite recursion as follows:

\begin{definition}\label{definition:OracleRuleWithSecondOrderVariable}
Let $\Lambda$ be a second-order variable that will be used to denote a well-order and $T$ be a formal theory. Define ${\rm Iter}_T(\phi,P)$ to be the formula
\[\nc_T\phi\vee \exists \psi \ \big ( \OmegaRule(P,\psi)\wedge \nc_T(\psi\to \phi) \big ).\]
Then, define
\begin{align*}
[\Lambda]_T\phi&\equiv \forall P \big( {\rm TR}_{{\rm Iter}_T}(P,\Lambda)\rightarrow \exists\lambda {<} \Lambda \ (\phi\in P_\lambda)\big);\\
\RecP {}T\phi&\equiv\exists\Lambda \ \big ( \wo(\Lambda)\wedge [\Lambda]_T\phi\big).
\end{align*}
As before, write $\RecPX{}T{\bm A}\phi (\dot{\bm A})$ instead of $\RecP{}{T|{\bm A}}\phi$, and for a set of formulas $\Gamma$, define $\PRFN\ite{ }\Gamma[T]$ analogously to Definition \ref{DefRFNProof}.
\end{definition}
This form of reflection gives rise to an axiomatization of $\atr$ \cite{Cordon2017}:

\begin{theorem}\label{TheoATR}
Let $ T$ be a c.e.~theory such that $\eca\subseteq T$ and $\atr$ proves that any set
$X$ can be included in a full $\omega$-model for $T$. Let $\Phi$ be any set of formulas such that $\{{\tt 0=1}\}\subseteq\Phi \subseteq {\bm \Pi}^1_2$. Then,
\begin{equation*}
\atr \equiv \eca + \PRFN\ite{ }{\Phi}[T].
\end{equation*}
\end{theorem}

In Theorem \ref{TheoATRBi}, we will extend this result to reflection over higher complexity classes, and show that it also gives rise to an axiomatization of transfinite induction.

\subsection{Formalizing $\omega$-logic via closed sets}

We can also characterize $\omega$-logic `from above' by observing that $\vecgamma$ is a theorem of $\omega$-logic if and only if it belongs to every set that is closed under its rules and axioms. To this end, define a formula
\[{\rm Der}^\rho_T(\vecgamma,Q) \equiv \nc_T \vecgamma \vee \exists {\bm x}\subseteq Q \ \RulePred ^\rho ({\bm x},\vecgamma)  \vee \OmegaRule(Q,\vecgamma),\]
stating that $\vecgamma$ is derivable from $Q$.
Then, define
\[{\rm Closed}^\rho_T (Q) \equiv \forall \vecgamma \big ({\rm Der}^\rho_T (\vecgamma, Q) \rightarrow \vecgamma \in Q \big ).\]

\begin{definition}
Given a theory $T$, define a formula
\[
\ClosedP \rho T \vecgamma \equiv \forall P \ \big ({\rm Closed}^\rho_T ( P ) \rightarrow \vecgamma\in P\big ),
\]
and write $\ClosedPX\rho T {{\VarOne}}\vecgamma (\dot{\bm X})$ instead of $\ClosedP\rho{T|{{\VarOne}}} \vecgamma$.
\end{definition}

\begin{lemma}\label{LemmCompClosOthers}
For any formula $\phi({\bm z},{\bm X})\in {\bm \Pi}^1_\omega$ with all free variables among those shown and $\rho\leq\omega$, $\eca$ proves that
\begin{enumerate}

\item $ \forall \bm X\ \forall \bm z \ \Big (\ProofPX \rho  T{\bm X}\phi
(\dot{\bm z},\dot{\bm X})\to \ClosedPX \rho T{\bm X}\phi (\dot{\bm z},\dot{\bm X}) \Big ),$

\item $ \forall \bm X\ \forall \bm z \ \Big ( \ClosedPX \rho  T{\bm X}\phi
(\dot{\bm z},\dot{\bm X}) \to \ModelsPX \rho T{\bm X}\phi (\dot{\bm z},\dot{\bm X}) \Big ).$

\end{enumerate}

\end{lemma}

\proof
Fix a tuple of numbers $\bm z$ and a tuple of sets $\bm X$, and let $m$ be the length of $\bm X$.
For the first item, let $\langle S,L\rangle $ be an $\omega$-proof of $\phi
(\bar {\bm z} , {\bm C}_{<m} )$ and let $ Q $ be closed under $\omega$-logic, both with an oracle for $\bm X$.
Then the set $\{\bm s \in S: L(\bm s)\not \in Q\}$, which is available in $\eca$, cannot have a minimal element, hence must be empty.
Since $Q$ was arbitrary, we obtain $\ClosedPX \rho T {\bm X}\phi (\bar{\bm z},\dot {\bm X})$.

For the second item, assume $\ClosedPX \rho  T{\bm X}\phi
(\bar {\bm z},\dot{\bm X})$ and let $\mathfrak M $ be an $\omega$-model with satisfaction class $S$ of rank $\rho$ and $|\mathfrak M|_{<m} = \bm X$.
Then, it is readily seen that the set $S$
is closed under the $\omega$-rule and cuts of rank $\rho$, hence $ \phi
(\overline{\bm z} ,{\bm C}_{<m} ) \in Q$ and thus $\model \models \phi
(\bar {\bm z},{\bm C}_{<m})$.
\endproof

\subsection{Formalizing $\omega$-logic via a least fixed point}

Finally, we may consider a formalization of $\omega$-logic via an inductively defined fixed point, rather than its approximations from above or below.

\begin{definition}
Fix a theory $T$, possibly with oracles, and $\rho\leq \omega$. Then, define a formula
\begin{align*}
\spc TQ \ & \equiv \ Q=\mu P. \ {\rm Der}^\rho_T ( n ,Q) .
\end{align*}
If $\spc TQ$ holds we will say that $Q$ is a {\em saturated provability class of rank $\rho$} ($\rho$-SPC) for $T$.
\end{definition}

With this, we may define our fixed point provability operator.

\begin{definition}
We define a formula
\[
\FixP \rho T\vecgamma \equiv \forall P \ \big (\spc T P \rightarrow \vecgamma\in P\big ).
\]
As before, we will write $\FixPX\rho T {{\VarOne}}\vecgamma (\dot{\bm X})$ instead of $\FixP\rho{T|{{\VarOne}}} \vecgamma.
$
\end{definition}

We will often want to apply this operator to formulas rather than sequents; when this is the case, we will identify a formula $\phi$ with the singleton sequent $\langle\phi\rangle$, and write $\FixPX\rho T {{\VarOne}}\phi$ instead of $\FixPX\rho T {{\VarOne}}\langle\phi\rangle$. Since SPC's are defined via an inductive definition, their existence can be readily proven in $\pica$.

\begin{lemma}\label{LemmPICAExists}
Let $T$ be any theory and $\rho\leq\omega$. Then, it is provable in $\pica$ that for every tuple of sets ${\ConOne}$ there exists a set $P$ such that $\spc {T|{\ConOne}}P$ holds.
\end{lemma}

\begin{proof}
Immediate from Lemma \ref{LemmExistsFix}.
\end{proof}

It is important to note that we have defined $\FixPX\rho T {{\VarOne}}\vecgamma$ by quantifying universally over all SPCs, so that $\Neg\FixPX\rho T {{\VarOne}} \vecgamma$ quantifies existentially over them. This means that such consistency statements automatically give us a bit of comprehension:

\begin{lemma}\label{LemmExistSPC}
If $T$ is any theory and $\vecgamma$ any sequent, then
\[\eca\vdash \forall {{\VarOne}}\ \big(\Neg\FixPX\rho T{{\VarOne}} \vecgamma (\dot{\bm X}) \ \to \ \exists P\ \spc {T|{{\VarOne}}}P\big).\]
\end{lemma}

However, this instance of comprehension by itself does not necessarily carry additional consistency strength, in the following sense:

\begin{lemma}\label{theorem:AddingExistenceIPCsRemainsEquiconsistent}
If $T$ is a theory extending $\rca$, and $\rho \leq \omega$,
\[
T \equiv_{\Pi^0_1} T + \forall {{\VarOne}}\, \exists P\ \spc{T|{{\VarOne}}} P;
\]
that is, the two theories prove the same $\Pi^0_1$ sentences.
\end{lemma}

This is proven in \cite{FernandezJoosten:2013:OmegaRuleInterpretationGLP} for a weaker notion of provability, but the argument carries through in our setting.
Roughly, we observe that $T + \Box_T \bot \equiv_{\Pi^0_1} T$, but $T + \Box_T\bot \ \vdash \ T + \forall {{\VarOne}}\, \exists P \ \spc{T|{{\VarOne}}} P$, since in this case an SPC would simply consist of the set of all sequents.
Note that for $\rho<\omega$ we use the finitary cut-elimination theorem, available in $\rca$, to eliminate cuts not applied to axioms from a proof of contradiction in $T$.

Unlike the existence of SPCs, their {\em uniqueness} is immediate from their definition.

\begin{lemma}\label{LemmUniqueSPC}
If $T$ is any theory and $\rho\leq\omega$, we have that
\[
\eca \vdash \forall {{\VarOne}} \ \exists_{\leq 1}P\ \spc {T|{{\VarOne}}} P,
\]
where $\exists_{\leq 1}P\, \phi(P)$ is an abbreviation of $\forall P \ \forall Q \ \big(\phi(P)\wedge\phi(Q)\to P = Q\big)$.
\end{lemma}

As one might expect, adding new sets to our oracle gives us a stronger theory:

\begin{lemma}\label{OracleMonotonicity}
Let $T$ be any theory and $\rho\leq\omega$. It is provable in $\eca$ that if $\ConOne$ is a tuple of sets and there exists an SPC for $T|{\ConOne}$, then for any sequent $\vecgamma$ and any set $Y$,
\[\FixPX \rho T {{\ConOne}}\vecgamma (\dot{\ConOne}) \to \FixPX \rho T {{\ConOne},B}\vecgamma (\dot{\ConOne}, \dot Y).\]
\end{lemma}

\begin{proof}
Suppose that $\FixPX \rho T {{\ConOne}}\vecgamma(\dot{\ConOne})$. Using our assumption, we may choose an SPC $P$ for $T|{{\ConOne}}$, so that $\vecgamma\in P$. Let $Q$ be an arbitrary SPC for $T|{{\ConOne}}, Y $. Observe that $Q$ contains all axioms of $T|{\ConOne}$ and is closed under all of its rules, so that by the minimality of $P$, we have that $P\subseteq Q$ and thus $\vecgamma \in Q$. Since $Q$ was arbitrary, it follows that $\FixPX \rho T {{\ConOne},B}\vecgamma (\dot{\ConOne}, \dot Y)$, as needed.
\end{proof}

Obseve also that our least-fixed-point formalization of $\omega$-provability is at least as strong as the formalization using $\omega$-proofs.
The next lemma follows easily from Lemma \ref{LemmCompClosOthers}.

\begin{lemma}\label{LemmCompTreeFix}
Given any formula $\phi$ and $\rho\leq\sigma\leq \omega$,
\[\eca \vdash \forall {\bm A} \forall {\bm n} \ \big ( \ProofPX \cpar T{\bm A}\vecgamma (\dot{\bm A}) \rightarrow \FixPX\sigma T {\bm A}\vecgamma (\dot{\bm A})\big ) .\]
\end{lemma}

Our goal now is to prove impredicative reflection within $\pica$. The following is a first approximation: $\pica$ proves that any formula proven in $\omega$-logic with oracles is true in any $\omega$-model.

\begin{lemma}[$\omega$--model soundness]\label{LemmOmSoundACA} Given any theory $T$, formula $\phi({\bm z},{\bm X})$ with all free variables among those shown, $a = |\bm X|$ and $\rho \leq \omega$,
\begin{enumerate}

\item\label{ItOnSOundOne} $\eca\vdash \forall P \ \forall  \bm X \ \forall \bm z \ \Big ( \spc {T|\bm X}P\wedge \ulcorner \phi(\dot{\bm z},{\bm C}_{<a})\urcorner \in P\to \ModelsPX\rho T{\bm X}\phi(\dot {\bm z},{\dot {\bm X}}) \Big );$

\item\label{ItOnSOundTwo} $\pica\vdash \forall \bm X  \ \forall \bm z \ \Big ( \FixPX\rho T{{\bm X}}\phi(\dot{\bm z},{\dot {\bm X}})\to \ModelsPX \rho T{{\bm X}}\phi(\dot{\bm z},{\dot {\bm X}}) \Big ).$

\end{enumerate}

\end{lemma}

\begin{proof}
For the first claim, reason in $\eca$. Let $\model$ be any model of $T$ of rank $\rho$ and let $P$ be a saturated provability class for $T|{\bm X}$ of rank $\rho$.
Then, $\sat_\model$ is closed under all the rules and axioms defining $P$, so that, by minimality, $P\subseteq \sat_\model$. It follows that if $\phi(\overline{\bm z},{\bm C}_{<a}) \in P$, then $\phi(\overline{\bm z},{\bm C}_{<a})\in \sat_\model$; that is, $\model\models \phi(\overline{\bm z},{\bm C}_{<a})$.

The second claim then follows from the first, together with the provable existence of a $\rho$-SPC in $\pica$.
\end{proof}

\section{Completeness and strong predicative reflection}\label{SecComp}

In this section we will recall some completeness results for formalized $\omega$-logic. It is well-known that $\omega$-logic is ${\Pi}^1_1$-complete \cite{Pohlers:2009:PTBook}, but it will be convenient to keep track of the second-order axioms needed to prove this. From these results, we will obtain a more general form of Theorem \ref{TheoATR}.

\subsection{Completeness results for $\omega$-logic}

We begin with a weak completeness result available in $\eca$.

\begin{lemma}\label{LemmComplete}
Fix a theory $T$ and $\rho\leq\omega$. Let $\vecgamma({\varThree},{\VarOne})\subseteq {\bm \Pi}^0_\omega$ be finite with all free variables among those shown. Then, it is provable in $\eca$ that

\begin{equation}\label{EqLemmComplet}
\forall {\bm X}\, \forall {\bm z}\, \Big(\bigvee\vecgamma ({\bm z},{\bm X})\rightarrow \FixPX \rho T {{\ConOne}}  \vecgamma( \dot{\bm z},{\dot{\bm X}}) \, \Big).
\end{equation}

\end{lemma}

\begin{proof}
Reasoning within $\eca$, fix a tuple $\bm z$ of natural numbers and $\bm X$ of sets, and let $a = |\bm X|$. Assume that $\bigvee\vecgamma({\bm z},{\bm X})$ holds, and write $\vecgamma=(\vecdelta,\phi)$ so that $\phi\in\vecgamma$ holds. We proceed by an external induction on $\phi$. Assume that $P$ is an arbitrary SPC for $T|\bm X$; we must prove that $\big ( \vecdelta,\phi(\bar{\bm z}, {\bm C}_{<a} ) \big )\in P$. If $\phi$ is atomic we use our standing assumption that the Tait calculus is formalized so that it is provably complete for atomic sentences.

Now assume that $\phi$ contains quantifiers. Let us consider the case where $\phi=\forall x \ \theta$. By the external induction hypothesis we have, for every $k$, that
\[\big (\vecdelta, \theta(\bar k,\bar {\bm z}, {\bm C}_{<a} )\big) \in P.\]
But, $P$ is closed under the $\omega$-rule, so we also have that
\[\big (\vecdelta, \forall x \, \theta (x,\bar {\bm z}, {\bm C}_{<a} )\big)\in P.\]

The remaining cases follow a similar structure; the case where $\phi$ is a Boolean combination of its subformulas is straightforward using the rules of the Tait calculus, and if $\phi=\exists x \ \theta(x)$, then for some $k$ we have that $\theta (\overline k)$ is true and we may use the induction hypothesis plus existential introduction.
\end{proof}

So, $\eca$ already proves the completeness of $\omega$-logic for arithmetical formulas, but we need to turn to $\aca$ to prove that it is also complete for ${\bm \Pi}^1_1$ formulas. The following is a mild modification of the Henkin-Orey $\omega$-completeness theorem \cite{GirardProofTheory,Orey56}.

\begin{theorem}\label{TheoCompTPro}
Let $\phi({\bm z},{\bm X})\in {\bm \Pi}^1_\omega$ be any formula with all free variables among those shown.
Then, for any $\rho\leq\omega$,
\[\aca\vdash \forall \bm X\ \forall \bm z \ \Big (\ModelsPX \rho  T{\bm X}\phi
(\dot{\bm z},\dot{\bm X})\to \ProofPX \rho T{\bm X}\phi (\dot{\bm z},\dot{\bm X}) \Big ).\]
If $\rnk \phi<\rho$, this is already provable in $\eca$.
\end{theorem}

\proof[sketch]
First we work in $\aca$ and consider the case $\rho = 0$.
Reason by contrapositive, and assume that $\ProofPX \rho T{\bm X}\phi (\dot{\bm z},\dot{\bm X})$ fails.
We construct a proof-search tree in such a way that all formulas are eventually analyzed.
To do this, it is convenient to think of sequents as sequences rather than sets.
Going from the root up, we always analyze the first formula, then place it at the end.
For example, if the first formula is a disjunction we would obtain
\[\dfrac{\Delta,\alpha\vee\beta,\alpha,\beta}{\alpha\vee\beta,\Delta}.\]
The next formula to be analyzed going up the tree would be the first formula of $\Delta$.
In the case of an existential quantifier, we put
\[\dfrac{\Delta,\exists x\alpha(x),\alpha(\overline n)}{\exists x \alpha(x),\Delta},\]
where $n$ is the least natural number such that $\alpha(\overline n)$ does not appear in $\Delta$.

Since $\ProofPX \rho T{\bm X}\phi (\dot{\bm z},\dot{\bm X})$ fails, the proof-search tree has an infinite branch, say $B$.
We define a model $\mathfrak M $ by setting $n\in |\mathfrak M|_i$ iff $ \overline n \notin V_i   $ appears in $B$ (recall that $(V_i)_{i<\omega}$ enumerates all second-order variables), and $\psi(\bm C)\in S_\mathfrak M$ iff ${\sim}\psi(\bm V) \in B$.
Note that the construction of $\mathfrak M$ uses ${\bm \Sigma}^0_1$-comprehension.
By the way the proof-search tree was constructed, it is not hard to check that $S_\mathfrak M$ satisfies the Tarski conditions, hence $\mathfrak M$ is an $\omega$-model.

For $\rho>0$, we modify the proof-search tree following a technique found in \cite{GirardII}.
On odd steps, we proceed as in the case for $\rho=0$.
However, on step $2m$, if $m=\ulcorner\psi\urcorner$  for some formula $\psi$ with rank less than $\rho$, we realize the following derivation.
Let $\partial_\psi$ be a standard proof of the $\psi\vee{\sim}\psi$ in the Tait calculus.
Then, continue the proof-search tree by
\[
\dfrac{\dfrac{\partial_\psi}{\psi\vee{\sim}\psi} \ \ \ \ \ \
\begin{array}{c}
\phantom{a}\\
{\sim}\psi\wedge\psi ,\Delta
\end{array}
}{\Delta}.
\]
This ensures that for any $\psi$ of rank less than $\rho$, either $\psi$ or ${\sim}\psi$ will appear in $B$, and hence $\mathfrak M$ will have rank at least $\rho$.

For the second claim we also use the modified proof-search tree.
Note that the proof of $\psi\vee{\sim}\psi$ is elementary in $\psi$ (see e.g.~\cite{Pohlers:2009:PTBook}), so the tree can be constructed in $\eca$.
$\mathfrak M$ can also be constructed in $\eca$, as all formulas appearing in $B$ have rank less than $\rho$, so $\psi$ appears in $B$ if and only if $\psi$ appears in $B$ by stage $2\ulcorner\psi\urcorner$.
That the Tarski conditionals hold can then be checked as above.
\endproof

The following is then immediate from Lemma \ref{LemmCompTreeFix}:

\begin{corollary}\label{CorCompTFix}
For any formula $\phi({\bm z},{\bm X})\in {\bm \Pi}^1_\omega$ and any $\rho\leq\omega$,
\[\aca\vdash \forall \bm X\ \forall \bm z \ \Big (\ModelsPX \rho  T{\bm X}\phi
(\dot{\bm z},\dot{\bm X})\to \FixPX \rho T{\bm X}\phi (\dot{\bm z},\dot{\bm X}) \Big ).\]
If $\rnk\phi<\rho$, the above is already provable in $\eca$.
\end{corollary}

For formulas of relatively low complexity, we can replace $\ModelsPX \rho  T{\bm A}\phi$ by $\phi$:

\begin{corollary}\label{CorComplete}\
Let $\rho\leq\omega$.

\begin{enumerate}

\item\label{ItCompleteOne} Given $\phi(\varThree,\VarOne)\in{\bm \Pi}^1_1$ with all free variables shown,
\[\aca \vdash \forall {{\bm X}} \, \forall \bm z \, \Big ( \phi(\bm z, {{\bm X}}) \to \ \FixPX\rho T{{{\ConOne}}} \phi(\dot{\bm z}, {{\dot{\bm X}}}) \Big ).\]

\item\label{ItCompleteTwo} Given $\phi(\varThree,\VarOne)\in{\bm \Sigma}^1_2$ with all free variables shown,
\[\aca \vdash \forall {{\bm X}} \, \forall \bm z \, \Big ( \phi(\bm z, {{\bm  X}}) \to \exists Y \ \FixPX\rho T{{{\bm X}},Y} \phi(\dot{\bm z}, {{\dot{\bm X}}}, \dot Y) \Big ) .\]

\end{enumerate}
If $\rnk\phi<\rho$, the above are already provable in $\eca.$
\end{corollary}

\begin{proof}
The first clam is immediate from Lemma \ref{LemModSound} and Corollary \ref{CorCompTFix}. For the second, suppose that $\phi(\bm z,\bm X)=\exists Y\psi(\bm z,{{\bm X}},Y)$, with $\psi\in \Pi^1_1({{\bm X}},Y)$. Then, if $\phi(\bm z,\bm X)$ holds we can fix $Y_0$ so that $\psi(\bm z,{{\bm X}},Y_0)$ is the case, and we may use the first claim to conclude that $\FixPX\rho T{{{\ConOne}},B} \psi(\bar{\bm z}, {{\dot{\bm X}}},\dot Y_0)$. By existential introduction we have $\FixPX\rho T{{{\ConOne}},B} \phi(\bar{\bm z},{{\dot{\bm X}}}, \dot Y_0)$.
\end{proof}

\subsection{Provable equivalences between formalizaitons}

The various formalizations of $\omega$-logic we are considering are equivalent.
In this section we will discuss in which theories the various equivalences can be proven.
Below, recall that the {\em Kleene-Brouwer ordering,} which we denote $\unlhd$, is defined on $\mathbb N^{< \omega}$ by setting $\bm s\unlhd\bm t$ if one of the following occurs:
\begin{enumerate*}[label=(\alph*)]
\item $\bm t\peq\bm s$, or

\item $\bm s,\bm t$ are incomparable under $\peq$, and for the least $i$ such that $ ({\bm s})_i\not= ({\bm t}) _i$, we have that $({\bm s})_i\leq  ({\bm t}) _i$.

\end{enumerate*}
It is provable in $\aca$ that $\unlhd$ is a well-order on $S$ whenever $S$ is well-founded \cite{Simpson:2009:SubsystemsOfSecondOrderArithmetic}.

\begin{theorem}\label{TheoEquivATR}
Let $T$ be any theory and $\rho \leq \omega$.

\begin{enumerate}

\item $\aca$ proves that, for every set-tuple $\bm A$ and every formula $\phi$, $\ProofPX \rho T {\bm A}\phi (\dot{\bm A})$, $\ModelsPX \rho T {\bm A}\phi (\dot{\bm A})$, and $\ClosedPX \rho T {\bm A}\phi (\dot{\bm A})$ are equivalent.
These equivalences are provable in $\eca$ when $\rnk\phi<\rho$.

\item $\atr$ proves that for every set-tuple $\bm A$ and every formula $\phi$, the above are also equivalent to $\RecPX {\omega} T {\bm A}\phi (\dot{\bm A})$.

\item $\pica$ proves that the above notions are morover equivalent to $\FixPX {} T {\bm A}\phi (\dot{\bm A})$.

\end{enumerate}

\end{theorem}

\proof
That $\ModelsPX \rho T{\bm A} \vecgamma (\dot{\bm A})$ implies
$\ProofPX \rho T {\bm A} \vecgamma (\dot{\bm A})$ is Theorem \ref{TheoCompTPro}.
That $\ProofPX \rho T {\bm A} \vecgamma (\dot{\bm A})$ implies $\ClosedPX \rho T {\bm A} \vecgamma (\dot{\bm A})$ and $\ClosedPX \rho T {\bm A} \vecgamma (\dot{\bm A})$ implies $\ModelsPX \rho T {\bm A} \vecgamma (\dot{\bm A})$ is Lemma \ref{LemmCompClosOthers}.
The third item then follows easily using Lemmas \ref{LemmCompTreeFix} and \ref{LemmOmSoundACA}.

For the second item, reason in $\atr$.
By Lemma \ref{LemModPar}.\ref{itModParATR}, $\ModelsPX {\rho} T {\bm A} \vecgamma (\dot{\bm A})$ is equivalent to $\ModelsPX { } T {\bm A} \vecgamma (\dot{\bm A})$.
That $\RecPX{ }T{\bm A}\vecgamma  (\dot{\bm A})$ implies $\ModelsPX { } T {\bm A} \vecgamma (\dot{\bm A})$ is proven in \cite{Cordon2017}, hence also $\RecPX{ }T{\bm A}\vecgamma  (\dot{\bm A})$ implies $\ModelsPX {\rho} T {\bm A} \vecgamma (\dot{\bm A})$.
Thus it remains to show that $\ProofPX \rho T {\bm A} \vecgamma (\dot{\bm A})$ implies $\RecPX{ }T{\bm A}\vecgamma  (\dot{\bm A})$.
Reasoning in $\atr$, suppose that $\langle S,L\rangle$ is an $\omega$-proof of $\phi$.
Let $\unlhd$ be the Kleene-Brouwer ordering. 
Since $S$ is well-founded, $\unlhd$ is a well-order on $S$.
Using arithmetical transfinite recursion, let $P$ be an IPC for $T|{\bm A}$ along $\langle S,{\unlhd}\rangle$. A straightforward transfinite induction along $\unlhd$ shows that, for all $\bm s\in S$, $\bigvee L(\bm s)\in P_{\bm s}$; in particular, $\phi \in P_{\langle\rangle}$. Since $P$ was arbitrary, we conclude that $\RecPX{}T{\bm A}\phi (\dot{\bm A})$.
\endproof

Note that as a special case of the first item, $\ProofPX {} T {\bm A}\phi (\dot{\bm A})$, $\ModelsPX {} T {\bm A}\phi (\dot{\bm A})$, and $\ClosedPX {} T {\bm A}\phi (\dot{\bm A})$ are provably equivalent in $\eca$.

\subsection{Predicative reflection and transfinite induction}

The above equivalences allow us to prove variants of Theorems \ref{TheoArai} and \ref{TheoModRFN} in terms of $\omega_\clsd$-reflection.

\begin{theorem} \

\begin{enumerate}

\item $\rca + \PRFN\clsd {}{{{\bm \Pi}^1_\omega}} \equiv {{\bm \Pi}^1_\omega}\text{-}\piti.$

\item For any $0<n\leq \omega$,
\[\aca+ \PRFN\clsd 0 {{\bm \Sigma}^1_{1+n}}[\aca]  \equiv {{\bm \Pi}^1_n}\text{-}\piti.\]

\end{enumerate}
\end{theorem}

\proof
We prove the first claim, as the proof of the second is analogous.
By Theorem \ref{TheoArai} we have that ${{\bm \Pi}^1_\omega}\text{-}\piti $ is equivalent to $\rca + \PRFN\tre {}{{{\bm \Pi}^1_\omega}}$.
It follows from Lemma \ref{LemmCompClosOthers} that
\[  \rca + \PRFN\clsd {}{{{\bm \Pi}^1_\omega}}\supseteq \rca + \PRFN\tre {}{{{\bm \Pi}^1_\omega}} \equiv {{\bm \Pi}^1_\omega}\text{-}\piti. \]
Meanwhile, by Proposition \ref{propRelBetTheos}, $\aca \subseteq {{\bm \Pi}^1_\omega}\text{-}\piti $, while by Theorem \ref{TheoEquivATR}, $\ProofPX \rho  T{\bm A}\phi
(\dot{\bm n},\dot{\bm A})$ and $\ClosedPX \rho T{\bm A}\phi (\dot{\bm n},\dot{\bm A})$ are provably equivalent in $\aca$, and therefore
\[{{\bm \Pi}^1_\omega}\text{-}\piti \equiv \aca + \PRFN\tre {}{{{\bm \Pi}^1_\omega}} \supseteq \rca + \PRFN\clsd {}{{{\bm \Pi}^1_\omega}}.\]
\endproof

We may also extend the results of \cite{Cordon2017} to reflection over formulas of higher complexity.

\begin{theorem}\label{TheoATRBi}
Let $U$ be a theory such that $\eca\subseteq U\subseteq \atr$. Then, for any $n\leq \omega$,
\begin{equation}\label{EqATRBI}
\atr +{\bm \Pi}^1_{n}\text{-}{\rm TI} \equiv U +\PRFN\ite{}{{{\bm \Sigma}^1_{1+n}}} [\aca] .
\end{equation}
\end{theorem}

\proof
The case for $n=0$ follows from Theorem \ref{TheoATR} in view of the fact that $\atr\vdash {\bm \Pi}^1_{0}\text{-}{\rm TI}$, so we assume $n>0$. Let $R\equiv U +\PRFN\ite{}{{{\bm \Sigma}^1_{1+n}}} [\aca]$.
Note that by Theorem \ref{TheoATR}, $\atr\subseteq R$, and hence $R\equiv\atr+\PRFN\ite{}{{{\bm \Sigma}^1_{1+n}}} [\aca]$. But, in view of Theorem \ref{TheoEquivATR},
\[R\equiv \atr+ \PRFN\mo{}{{{\bm \Sigma}^1_{1+n}}} [\aca]\equiv \atr+ \PRFN\mo{0}{{{\bm \Sigma}^1_{1+n}}}[\aca],\]
where the second equivalence is due to the fact that $\atr$ proves that any satisfaction class extends to a full satisfaction class.
By Theorem \ref{TheoModRFN},
\[\atr+ \PRFN\mo{0}{{{\bm \Sigma}^1_{1+n}}}[\aca]\equiv \atr+{\bm\Pi}^1_n\text{-}{\rm TI},\]
as needed.
\endproof

In view of Proposition \ref{propRelBetTheos}, it follows that Theorem \ref{TheoATR} is sharp:

\begin{corollary}
$\atr\not\vdash \PRFN\ite{}{{\bm \Sigma}^1_{2}}[\aca]$.
\end{corollary}

\begin{remark}
We could instead use Theorem \ref{TheoArai} to obtain a variant of Theorem \ref{TheoATRBi} with the pure Tait calculus in place of $\aca$. For greater generality, it may be of interest to analyze the proof in \cite{Jager1999} to identify the minimal requirements on a theory $T$ which would allow us to replace $\aca$ by $T$.
\end{remark}

\section{Consistency and reflection using inductive definitions}\label{section:OracleConsistencyAndOracleReflection}

In this section we will define the notions of reflection and consistency that naturally correspond to $\FixPX \rho T{\bm A}$. Moreover, we will link the two notions to each other and see how they relate to comprehension. Below, recall that $\bot$ denotes the empty sequent.

\begin{definition} 
Given a theory $T$, $\rho\leq\omega$, and a set of formulas $\Phi$, we define the schemas
\begin{align*}
\PRFN\fx\rho \Phi [T]&=\forall {{\ConOne}} \, \forall {{\varOne}}\ \Big(\, \FixP\rho T \, \phi(\dot{{\varOne}},\dot {\ConOne}) \to \phi({{\varOne}},{{\ConOne}}) \Big ),\\
\PCONS\fx\rho \Phi [T]&=\forall {{\ConOne}} \, \forall {\varOne}\ \Neg \Big(\, \FixP\rho T  \, \phi (\dot{{\varOne}},\dot {\ConOne}) \wedge \FixP \rho T  \, \Neg \phi(\dot{{\varOne}},\dot {\ConOne}) \Big ),\\
\PCons\fx\rho[T]&=\forall {\ConOne} \Neg \FixP\rho T   \bot ( \dot {\ConOne} ),
\end{align*}
for $\phi({{\varThree}},{{\VarOne}}) \in \Phi$ with all free variables among those shown.
\end{definition}

\begin{lemma}\label{LemmEquivRFN} Given any theory $T$ and set of formulas $\Phi$,
\begin{enumerate}

\item if $\rho\leq \omega$, $\aca+\PRFN\fx \rho \Phi [T]\vdash \PRFN \mo \rho \Phi [T]$;

\item if $\rho\leq \omega,$ $\pica + \PRFN \fx \rho \Phi \equiv \pica + \PRFN\mo \rho \Phi [T].$

\end{enumerate}

\end{lemma}

\proof
For the first claim, reason in $\aca+\PRFN\fx \rho \Phi [T]$. Suppose that $\phi\in \Phi $ and $\ModelsP \rho T\phi(\bar{\bm z},\dot {\bm X})$. Then, by Corollary \ref{CorCompTFix}, $\Prop \rho\fx T\phi(\bar {\bm z},\dot{\bm X})$, and thus $\phi(\bm z,\bm X)$ holds by $\PRFN \fx \rho \Phi$. For the second claim, the remaining inclusion follows from Lemma \ref{LemmOmSoundACA}.
\endproof

Of course, the schema $\PCONS\fx\rho\Gamma[T]$ is only interesting when $\rho<\omega$, since otherwise it is just equivalent to consistency.

\begin{lemma}\label{LemmConsImpCONS}
If $T$ is any theory and $\rho\leq\omega$, then
\[\eca + \PCONS\fx\rho{{{\bm \Pi}^1_\omega}}[T]\subseteq
\eca+\PCons\fx\omega[T].\]
\end{lemma}

\begin{proof}
Reasoning by contrapositive, if $\PCONS\fx\rho{{{\bm \Pi}^1_\omega}}[T]$ fails, then for some formula $\phi({\varThree},{\VarOne})$, some tuple of sets $\ConOne$ and some tuple of natural numbers $\varOne$, we have that
\[\FixP\rho T  \, \phi(\bar{\varOne},\dot {\ConOne} ) \wedge \FixP\rho T  \, \Neg\phi(\bar{\varOne}, \dot {\bm X} ),\]
which applying one cut gives us $\Prop \omega\fx  T  \bot ( \dot {\ConOne} )$.
\end{proof}

Let us now see that with just a little amount of reflection we get arithmetical comprehension. The fist step is to build new sets out of our provability operators.

\begin{lemma}\label{LemmTildeExists}
Let $T$ be any theory, $\phi(z,\VarOne)$ be any formula and $\rho\leq\omega$. Then,
\[
\eca \vdash \forall {{\ConOne}} \, \exists W \ \forall n \ \Big ( n \in W \leftrightarrow \Prop \rho\fx T \phi(\dot n,\dot {{\ConOne}})\Big ).
\] 
\end{lemma}

\begin{proof} 
Reason within $\eca$. Pick a tuple of sets ${{\ConOne}}$ and let $a$ be the length of $\bm X$. Consider two cases; if there does not exist a $\rho$-SPC for $T|{{\ConOne}}$, then we may set $W=\mathbb N$ and observe that $\forall n \ \big ( n \in W \leftrightarrow \Prop \rho\fx  T  \phi(\dot n,\dot {{\ConOne}} ) \big )$ holds trivially by vacuity.
If such an SPC does exist, by Lemma \ref{LemmUniqueSPC} it is unique; call it $P$. Within $\eca$ we may form the set
\[W=\{n : \phi(\bar n,{\bm C}_{<a}) \in P\}.\]
Then, if $n \in W$ is arbitrary we have by the uniqueness of $P$ that $\Prop \rho\fx  T \phi(\bar n,\dot {{\ConOne}})$ holds. Conversely, if $\Prop \rho\fx  T\phi(\bar n, \dot {\bm X})$ holds, then in particular $\phi(\bar n,{\bm C}_{<a}) \in P$ and $n\in W$ by definition, so $W$ has all desired properties.
Since ${{\ConOne}}$ was arbitrary, the claim follows.
\end{proof}

\begin{lemma}\label{LemmRefComp}
Let $T$ be any theory and $\rho\leq\omega$. Then,
\[\aca\subseteq\eca+\PRFN\fx\rho{{{\bm \Sigma}^0_1 }}[T].\]
\end{lemma}

\begin{proof}
Work in $\eca+\PRFN\fx\rho{{{\bm \Sigma}^0_1 }}[T]$. We only need to prove ${\bm\Sigma}^0_1$-$\rm CA$, that is,
\[\forall {{\VarOne}}\, \exists Y \, \forall n\ \big (n \in Y \, \leftrightarrow \, \phi(n,{{\VarOne}} ) \big ),\]
where $\phi(n,{{\VarOne}})$ can be any formula in $\Sigma^0_1({{\VarOne}})$.

Fix some tuple of sets ${\ConOne}$. By Lemma \ref{LemmTildeExists}, we can form the set
\[
Z=\{n : \Prop \rho\fx T \ \phi(\bar n, \dot {\bm X})\}.
\]
We claim that $\forall n \ \big ( n\in Z \leftrightarrow \phi(n, {{\ConOne}}) \big )$ which finishes the proof. If $n\in Z$, then, by reflection, $\phi(n,{{\ConOne}})$. On the other hand, if $\phi(n,{{\ConOne}})$ we get by arithmetical completeness (Lemma \ref{LemmComplete}) that $\Prop \rho\fx T \phi(\bar n, \dot {\bm X})$, so that $n\in Z$.
\end{proof}

The above result along with the completeness theorems mentioned earlier may be used to prove that many theories defined using reflection and consistency are equivalent. Below, $\Neg \Phi =\{\Neg\phi:\phi\in \Phi \}$.

\begin{lemma}\label{LemConsVsRFN}
Let $T$ be a theory extending $\Robinson$, $\Phi$ a set of formulas and $\rho\leq\omega$. Then:
\begin{enumerate}
\item if ${\bm \Sigma}^0_ 1\subseteq \Phi\subseteq{\bm \Pi}^1_1$,
\[\eca+\PCONS\fx \rho \Phi [T] \equiv \eca + \PRFN \fx \rho { \Phi \cup \Neg \Phi }[T];\]

\item $\eca + \PCons\fx\omega [T]\equiv \eca+ \PRFN \fx\omega{{{\bm \Pi } ^1_2}}[T].$
\end{enumerate}
\end{lemma}

\begin{proof}
For the first claim, let us begin by proving that
\[\eca + \PCONS\fx \rho \Phi [T] \subseteq  \eca + \PRFN \fx \rho {\Phi\cup\Neg\Phi}[T].\]
Assume $\PRFN \fx \rho {\Phi\cup\Neg\Phi}[T]$ and let $\phi\in\Phi$. Towards a contradiction, suppose that for some tuple of natural numbers $\varOne$ and some tuple of sets $\ConOne$,
\[\FixPX\rho T {\ConOne} \, \phi(\bar{{\varOne}},{\dot{\bm X}}) \wedge \FixPX\rho T {{\ConOne}} \, \Neg\phi(\bar{{\varOne}},{\dot{\bm X}}).\]
By reflection, this gives us $\phi({{\varOne}},{{\ConOne}}) \wedge \Neg\phi( {{\varOne}},{{\ConOne}}),$ which is impossible. Since $\phi$ was arbitrary, the claim follows.

Next we prove that
\[\eca + \PCONS\fx \rho \Phi [T] \supseteq \eca + \PRFN \fx \rho {\Phi\cup\Neg\Phi}[T].\]
For this, fix $\phi \in \Phi\cup\Neg\Phi$ and reason in $\eca + \PCONS\fx \rho \Phi [T]$. We first consider the case where $\phi=\phi({{\varThree}},{{\VarOne}})$ is arithmetical. 

Let $\varOne$ be a tuple of natural numbers and $\ConOne$ a tuple of sets such that $\FixPX\rho T{{\ConOne}} \phi(\bar{\varOne},{\dot{\bm A}})$. 
If $\phi({\varOne}, {\ConOne})$ were false, by Lemma \ref{CorComplete}.\ref{ItCompleteOne}, we would also have that $\FixPX\rho T{{\ConOne}} \Neg\phi(\bar{\varOne},{\dot{\bm A}})$; but this contradicts $\PCONS\fx \rho \Phi [T]$. We conclude that $\phi({\varOne}, {\ConOne})$ holds, as desired.

Before considering the case where $\phi$ is not arithmetical, observe that since ${\bm \Sigma}^0_ 1\subseteq \Phi$, it follows that
\[\eca + \PCONS\fx \rho \Phi [T] \supseteq \eca + \PRFN \fx \rho {{\bm \Sigma}^0_1}[T],\]
and by Lemma \ref{LemmRefComp}, we have that
\[\aca\subseteq \eca + \PRFN \fx \rho {{\bm \Sigma}^0_1}[T] ,\]
so we may now use arithmetical comprehension.

With this observation in mind, the argument will be very similar to the one before. Once again, suppose that $\FixPX\rho  T{{\ConOne}} \phi( \bar {\varOne},{\dot{\bm X}})$ for some tuples ${\varOne},{{\ConOne}}$. If $\phi({\varOne},{{\ConOne}})$ were false, by Corollary \ref{CorComplete}.\ref{ItCompleteTwo}, there would be $Y$ such that $\FixPX\rho T{{{\ConOne}},B} \Neg\phi(\bar{\varOne},{\dot{\ConOne},\dot Y})$. By Lemma \ref{LemmExistSPC}, $\eca + \PCONS\fx \rho \Phi [T]$ implies that there exists a $\rho$-SPC for $T|\ConOne$, and hence we may use Lemma \ref{OracleMonotonicity} to see that
\[\FixPX\rho  T{{{\ConOne}},Y} \phi(\bar{\varOne},{\dot{\bm X},\dot Y})\wedge \FixPX\rho  T{{{\ConOne}},Y}\Neg\phi(\bar{\varOne},{\dot{\bm X},\dot Y}).\]
As before, this contradicts  $\PCONS\fx \rho \Phi [T]$. We conclude that $\phi({\varOne}, {{\ConOne}})$ holds, as desired.

Now we prove the second claim. The right-to-left implication is obvious, so we focus on the other. Reason in $\eca+\PCons\fx\omega [T]$. By Lemma \ref{LemmConsImpCONS}, this implies $\PCONS\fx\omega {{{\bm \Pi}^1_\omega}} [T]$, so that using Lemma \ref{LemmRefComp} and the previous item, we may reason in $\aca$.

Fix $\phi ({\varThree},{\VarOne}) \in{\bm \Pi}^1_2$ and assume that $\FixPX\omega T{{\ConOne}} \phi({\bar {\varOne}},{\dot{\bm X}})$. If $\phi({{\varOne}},{\ConOne})$ were false, then by Corollary \ref{CorComplete}, we would also have $\FixPX\omega  T{{\ConOne},Y} \Neg\phi({\bar {\varOne}},{\dot{\bm X}},\dot Y)$ for some set $Y$, and using Lemma \ref{OracleMonotonicity} as above,
\[ \FixPX\omega  T{{\ConOne},B} \phi({\bar {\varOne}},{\dot{\bm X}},\dot Y) \wedge \FixPX\omega  T{{\ConOne},B} \Neg\phi({\bar {\varOne}},{\dot{\bm X}},\dot Y).\]
But this contradicts $\PCONS\fx\omega {{{\bm \Pi}^1_\omega}} [T]$, and we conclude that $\phi ({\varThree},{\VarOne})$ holds.
\end{proof}

Next, we turn our attention to proving that reflection implies $\pica$. This fact will be an easy consequence of the following.

\begin{lemma}\label{LemmTildeToTR}
Let $T$ be any theory, $\rho\leq\omega$, $\Phi\subseteq {\bm \Pi}^0_\omega({{\VarOne}})$, and $\phi(\varThree,\VarOne)\in \Pi^1_1/\Phi$ with all free variables among those shown. Then, it is provable in $\aca +  \PRFN \fx\rho {{{\Pi}^1_1/\Phi}} [T]$ that
\[\forall \ConOne\, \forall \varOne \ \big ( \phi({{\varOne}},{{\ConOne}}) \leftrightarrow \FixPX \rho T{{\ConOne}} \phi(\dot {{\varOne}},{\dot{\bm X}}) \big ).\]
\end{lemma}

\begin{proof}
Reason in $\aca + \PRFN \fx\rho {{{\Pi}^1_1/\Phi}} [T]$ and let ${{\ConOne}}$ and ${{\varOne}}$ be arbitrary. For the left-to-right direction we see that if $\phi({{\varOne}},{\ConOne})$ holds, then by provable ${\bm \Pi}^1_1$-completeness (Corollary \ref{CorComplete}), $\FixPX \rho T{{\ConOne}} \phi(\bar {\varOne},{\dot{\bm X}})$ holds as well. For the right-to-left direction, if $\FixPX\rho T{{\ConOne}} \phi(\bar {{\varOne}},{\dot{\bm X}})$, by $\PRFN\fx\rho{{\Pi^1_1/\Phi}}[T]$, $\phi( {{\varOne}},{\bm X})$ holds.
\end{proof}

We can now finally combine all our previous results and formulate the main theorem of this section.

\begin{theorem}\label{TheoRFNATR}
Given any theory $T$,
\[\aca + \PRFN\fx\rho{{\picaform}} [T] \vdash \pica.\]
\end{theorem}

\begin{proof}
Work in $\aca + \PRFN\fx\rho{{\picaform}} [T].$ By Theorem \ref{TheoPicaSimp}, we need only prove comprehension for arbitrary $\phi(n,{\VarOne})\in {\Pi}^1_1/\Sigma^0_2 ({{\VarOne}})$.
Fix a tuple of sets $\ConOne$. 
By Lemma \ref{LemmTildeExists}, there is a set $W$ satisfying
\[\forall n \ \big ( n \in W \leftrightarrow \FixPX\rho T {{\ConOne}} \phi(\dot n,{\dot{\bm X}})\big ).\]
But by Lemma \ref{LemmTildeToTR}, this is equivalent to
\[\forall n \ \big ( n \in W \leftrightarrow   \phi(n, {{\ConOne}})\big ).\]
Since $\phi$ and ${{\ConOne}}$ were arbitrary, we obtain $\pica$, as desired.
\end{proof}

Thus impredicative reflection implies impredicative comprehension, as claimed. Next we will prove the opposite implication, but for this we will first need to take a detour through {\em $\beta$-models.}

\section{Countable $\beta$-models and impredicative reflection}\label{section:CountableCodedModels}

Our goal in this section is to derive a converse of Theorem \ref{TheoRFNATR}. The main tool for this task will be the
notion of a \emph{countable coded $\beta$--model}. In what follows we shall discuss the definition and basic existence results for such models.

Note that the converse of Lemma \ref{LemModSound} is not always true for ${\bm \Pi}^1_1$-sentences, as we are not truly quantifying over {\em all} subsets of $\mathbb N$. Nevertheless, for special kinds of models it may actually be the case that $\model\models \forall X\phi(X)$ implies that $\forall X\phi(X)$ when $\phi$ is arithmetical; such models are called {\em $\beta$-models.}

Below, recall that $\bm V=\langle V_i\rangle_{i\in\mathbb N}$ is assumed to be a sequence listing all second-order variables, and that ${\bm S}_{<a}=\langle S_i\rangle_{i<a}$ for any sequence $\bm S$.

\begin{definition}
A countable coded $\omega$-model $\mathfrak M$ is a {\em $\beta$-model} if for every $\phi (\varThree,{\bm V}_{<a})\in {\bm \Pi}^1_1$ and every $\varOne$, $\phi (\varOne,|{\mathfrak M}|_{<a})$ holds if and only if $\mathfrak M\models\phi(\bar\varOne,{\bm C}_{<a})$.
\end{definition}

Thus, $\beta$-models reflect ${\bm \Pi}^1_1$ formulas; however, with no additional assumptions, we can push this property a bit farther.

\begin{lemma}\label{LemmAbsolute}
Fix a formula $\phi(\varThree, {\bm V}_{<a} )\in {\bm \Sigma}^1_2$. It is provable in $\aca$ that, for all $a$-tuples $\bm W$ and all $\varOne$, if $\mathfrak M$ is a $\beta$-model with $|\model|_{<a}=\bm W$ and such that $\mathfrak M\models\phi (\bar\varOne,{\bm C}_{<a})$, then $\phi (\varOne,{\bm W})$ holds.
\end{lemma}

\begin{proof}
Write $\phi=\exists X \ \forall Y \ \psi(\varThree,{\bm V}_{<a},X,Y)$ and suppose that ${\bm W}$ is an $a$-tuple of sets and $\model$ a model with $|\model|_{<a}={\bm W}$. Then, if $\mathfrak M\models \phi({\bm C}_{<a})$, it follows that for some $m$, $\mathfrak M\models \forall Y \ \psi \ ({\bm C}_{<a},C_m, Y )$. But since by assumption $\mathfrak M$ is a $\beta$-model, it follows that $\forall Y \ \psi({\bm W} ,|\model|_m,Y)$ holds, hence so does $\phi=\exists X \ \forall Y \ \psi({\bm W},X,Y)$.
\end{proof}

A good part of the theory of $\beta$-models may be formalized within $\pica$. Theorems \ref{TheoPiti} and \ref{TheoBetaExists} may be found in \cite{Simpson:2009:SubsystemsOfSecondOrderArithmetic}. Recall that we defined the theories $\atr$ and ${\bm \Pi}^1_\omega\text{-}\piti$ in Section \ref{SubsecSSOA}.

\begin{theorem}\label{TheoPiti}
It is provable in $\atr$ that, for every countable coded $\beta$-model $\mathfrak M$, $\mathfrak M\models{\bm \Pi}^1_\omega \text{-}\piti$.
\end{theorem}


\begin{theorem}\label{TheoBetaExists}
It is provable in $\pica$ that for every $a$-tuple of sets ${\ConOne}$ there is a full $\beta$-model $\mathfrak M$ such that $|\model|_{<a}={\ConOne}$.
\end{theorem}

With these results in mind, we can now easily prove that comprehension implies reflection.

\begin{lemma}\label{LemmaATRRFNOne}

Let $U,T$ be theories such that $U$ extends $\aca$ and $\rho\leq\omega$. If $U$ proves that any $a$-tuple $\ConOne$ can be included in an $\omega$-model satisfying $T$ of rank $\rho $, then for any $\phi(\bm z,\bm V_{<a})\in {\bm \Pi}^1_2$ with all free variables shown, $U$ proves that
\begin{equation}\label{EqPReflects}
\forall P \ \forall \bm X \ \forall \bm z \ \Big ({\rm SPC}^\rho_T (P)\wedge \big (\ulcorner \phi(\dot{\bm z},\bm C_{<a}) \urcorner \in P\big)\rightarrow \phi(\bm z,\bm X)  \Big).
\end{equation}
If $U$ proves that any $a$-tuple $\ConOne$ can be included in a $\beta$-model satisfying $T$ of rank $\rho\leq\omega$, \eqref{EqPReflects} holds for $\phi\in {\bm \Pi}^1_3$.
\end{lemma}

\begin{proof}
For the first claim, let $\phi(\varThree,{\bm V}_{<a})=\forall Y \ \psi(\varThree,{{\bm V}_{<a}},Y )$, where $\psi\in {\bm \Sigma}^1_1$ with all free variables shown, and reason in $\aca$. Fix an $a$-tuple ${\ConOne}$ of sets, a tuple of natural numbers $\varOne$, and a $\rho$-SPC $P$, and assume that $ \phi(\bar\varOne,{\bm C}_{<a}) \in P$. Let $Y$ be arbitrary and $\mathfrak M$ be an $\omega$-model satisfying $T$ with $|\model|_{\leq a}= ( {\ConOne},Y ) $. 
Then, by Lemma \ref{LemmOmSoundACA}.\ref{ItOnSOundOne}, $\mathfrak M\models\psi(\bar\varOne,{\bm C}_{<a},C_{a})$, so that by Lemma \ref{LemmAbsolute}, $\psi(\varOne,{\ConOne},Y)$ holds. Since $Y$ was arbitrary, we conclude that $\phi(\varOne,\ConOne)=\forall Y \psi(\varOne,{\ConOne}, Y)$ holds. The second claim is similar, but we take $\psi\in {\bm \Sigma}^1_2$ and use Lemma \ref{LemmAbsolute}.
\end{proof}

Using the fact that $\pica$ proves the existence of a $\rho$-SPC, we obtain the following:

\begin{corollary}\label{CorATRRFNTwo}
If $\rho\leq \omega$ and $\pica$ proves that any $a$-tuple $\ConOne$ can be included in a $\beta$-model for $T$ of rank $\rho$, then
\[\pica\vdash \PRFN\fx\rho{{{\bm \Pi}^1_3}}[T].\]
\end{corollary}

We may now summarize our results in our main theorem.

\begin{theorem}\label{TheoMain}
Let $ T$ be a theory such that $\pica$ proves that any set-tuple $\ConOne$ can be included in a $\beta$-model for $T$. Let $\picaform\subseteq\Phi\subseteq {\bm \Pi } ^1_3$. Then, for any $\rho\leq \omega$,
\begin{align}\label{EqTheoMain}
\nonumber \pica & \equiv \eca + \PRFN\fx\rho {\Phi}[T] \\
& \equiv \eca + \PCONS \fx\rho {\Phi} [T] \equiv \eca + \PCons \fx\omega [T].
\end{align}

\end{theorem}

\begin{proof}
All inclusions are immediate from Lemmas  \ref{LemmRefComp} and \ref{LemConsVsRFN}, Theorem \ref{TheoRFNATR} and Corollary \ref{CorATRRFNTwo}. 
\end{proof}

\begin{corollary}\label{CorMain}
Let $\mathcal G=\{\taitc,\eca,\rcaa,\rca,\aca,\atr ,{\bm \Pi}^1_\omega\text{-}{\piti} \}$ and $\rho\leq\omega$. Then, \eqref{EqTheoMain} holds for any $T \in \mathcal G$.
\end{corollary}

In view of Theorem \ref{TheoModRFN}, we may extend these results to reflection over higher complexity classes.

\begin{theorem}\label{TheoMainBi}
For any $n \in [1,\omega]$ and $\rho\leq\omega$,
\begin{equation}\label{EqMainBI}
\pica +{\bm \Pi}^1_{n}\text{-}{\rm TI} \equiv \eca +\PRFN\fx\rho{{{\bm \Sigma}^1_{1+n}}} [\aca] .
\end{equation}
\end{theorem}

\proof
Let $B=\pica +{\bm \Pi}^1_{n}\text{-}{\rm TI}$ and $R= \eca +\PRFN\fx\rho{{{\bm \Sigma}^1_{1+n}}} [\aca]$. First we show that
$B \subseteq R$. Since $ \eca + \PRFN\fx\rho{{ \picaform }} [\aca]  \subseteq R$, we obtain
$\pica \subseteq R$. We have that
\[\pica + \PRFN\fx\rho{{{\bm \Sigma}^1_{1+n}}} [\aca]  \vdash \PRFN\mo \rho {{{\bm \Sigma}^1_{1+n}}} [\aca]\]
by Lemma \ref{LemmEquivRFN}. But, $\pica+\PRFN\mo \rho {{{\bm \Sigma}^1_{1+n}}} [\aca]\vdash \PRFN\mo \rho {{{\bm \Sigma}^1_{1+n}}} [\aca]$ by Lemma \ref{LemModPar}, and we obtain $R  \vdash {\bm \Pi}^1_{n}\text{-}{\rm TI}$
by Theorem \ref{TheoModRFN}.

Next we show that $R\subseteq B$. By Theorem \ref{TheoModRFN} and the fact that $\pica$ proves that any valuation can be extended to a full valuation, we have that $B \vdash \PRFN\mo \rho {{{\bm \Sigma}^1_{1+n}}} [\aca] $. But, by Lemma \ref{LemmEquivRFN},
\[\pica + \PRFN\mo \rho {{{\bm \Sigma}^1_{1+n}}} [\aca] \vdash \PRFN\fx \rho{{{\bm \Sigma}^1_{1+n}}} [\aca] .\]
\endproof

\begin{remark}
Note that $\PRFN\fx\rho\Phi[T]$ is equivalent to the conjunction of the following two:
\begin{enumerate}[label=(\arabic*)]

\item\label{ItRemOne} Every $\rho$-SPC contains only true formulas from $\Phi$,

\item\label{ItRemTwo} there exists a $\rho$-SPC containing any tuple of parameters $\bm A$.

\end{enumerate}
Thus it is tempting to conjecture that either \ref{ItRemOne} or \ref{ItRemTwo} is sufficient to obtain $\pica$. But this is not the case. Observe that $\aca$ proves that $\rca$ has $\omega$-models of any finite rank $\rho$ \cite[Lemma VII.2.2]{Simpson:2009:SubsystemsOfSecondOrderArithmetic}, hence by Lemma \ref{LemmaATRRFNOne}, it proves that any $\rho$-SPC for $\rca$ reflects ${\bm \Pi}^1_2$ formulas, yet $\aca\subsetneq\pica$. Similarly, $\atr$ proves that $\aca$ has full $\omega$-models \cite[Theorem VIII.1.13]{Simpson:2009:SubsystemsOfSecondOrderArithmetic}, so it proves that any $\omega$-SPC for $\aca$ reflects ${\bm \Pi}^1_2$ formulas. We conclude that \ref{ItRemOne} is not sufficient.

Meanwhile, by Lemma \ref{theorem:AddingExistenceIPCsRemainsEquiconsistent}, $T=\aca+\exists P \ {\rm SPC}^\rho_{\aca}(P)$ is equiconsistent with $\aca$, hence $T\subsetneq\pica$. It follows that \ref{ItRemTwo} is not sufficient either.

On the other hand, the reader may verify, using Lemma \ref{LemmaATRRFNOne}, that
\[\pica\equiv \aca+\exists P \ {\rm SPC}^0_{\rca}(P)\equiv \atr+\exists P \ {\rm SPC}^\omega_{\aca}(P).\]
\end{remark}




 \bibliographystyle{splncs04}
 \bibliography{References}

\begin{thebibliography}{10}
\providecommand{\url}[1]{\texttt{#1}}
\providecommand{\urlprefix}{URL }
\providecommand{\doi}[1]{https://doi.org/#1}

\bibitem{Arai1998}
Arai, T.: Some results on cut-elimination, provable well-orderings, induction
  and reflection. Annals of Pure and Applied Logic  \textbf{95}(1),  93 -- 184
  (1998)

\bibitem{Beklemishev:1997:InductionRules}
Beklemishev, L.D.: Induction rules, reflection principles, and provably
  recursive functions. Annals of Pure and Applied Logic  \textbf{85},  193--242
  (1997)

\bibitem{Boolos:1993:LogicOfProvability}
Boolos, G.: The {L}ogic of {P}rovability. Cambridge University Press, Cambridge
  (1993)

\bibitem{IteratedInductive}
Buchholz, W., Feferman, S., Pohlers, W., Sieg, W.: Iterated Inductive
  Definitions and Subsystems of Analysis: Recent Proof-Theoretical Studies.
  Lecture Notes in Mathematics, Springer (1981)

\bibitem{Cordon2017}
Cord\'on-Franco, A., Fern\'andez-Duque, D., Joosten, J.J., Lara-Mart\'{\i}n,
  F.: Predicativity through transfinite reflection. Journal of Symbolic Logic
  \textbf{82}(3),  787--808 (2017)

\bibitem{FernandezJoosten:2013:OmegaRuleInterpretationGLP}
Fern{\'{a}}ndez{-}Duque, D., Joosten, J.: The omega-rule interpretation of
  transfinite provability logic. Annals of Pure and Applied Logic
  \textbf{169}(4),  333--371 (2018)

\bibitem{Friedman75}
Friedman, H.: Some systems of second order arithmetic and their use. In:
  Proceedings of the International Congress of Mathematicians, Vancouver 1974.
  pp. 235--242 (1975)

\bibitem{GirardII}
Girard, J.Y.: Proof-theory and logical complexity {I}{I},
  \url{http://girard.perso.math.cnrs.fr/Archives4.html}, unpublished

\bibitem{GirardProofTheory}
Girard, J.Y.: Proof theory and logical complexity. Vol. 1. Studies in proof
  theory, Bibliopolis, Napoli (1987)

\bibitem{HajekPudlak:1993:Metamathematics}
H\'ajek, P., Pudl\'ak, P.: Metamathematics of {F}irst {O}rder {A}rithmetic.
  Springer-{V}erlag, Berlin, Heidelberg, New York (1993)

\bibitem{hirst_1999}
Hirst, J.: Ordinal inequalities, transfinite induction, and reverse
  mathematics. Journal of Symbolic Logic  \textbf{64}(2),  769–774 (1999)

\bibitem{Jager1999}
J{\"a}ger, G., Strahm, T.: Bar induction and $\omega$ model reflection. Annals
  of Pure and Applied Logic  \textbf{97},  221--230 (1999)

\bibitem{KreiselLevy:1968:ReflectionPrinciplesAndTheirUse}
Kreisel, G., L\'evy, A.: Reflection principles and their use for establishing
  the complexity of axiomatic systems. Zeitschrift f\"ur mathematische Logik
  und Grundlagen der Mathematik  \textbf{14},  97--142 (1968)

\bibitem{Lob:1955:SolutionProblemHenkin}
L\"{o}b, M.H.: Solution of a problem of {Leon Henkin}. Journal of Symbolic
  Logic  \textbf{20},  115--118 (1955)

\bibitem{Orey56}
Orey, S.: On $\omega$-consistency and related properties. J. Symb. Log.
  \textbf{21}(3),  246--252 (1956)

\bibitem{Pohlers:2009:PTBook}
Pohlers, W.: Proof Theory, The First Step into Impredicativity.
  Springer-Verlag, Berlin Heidelberg (2009)

\bibitem{Simpson82}
Simpson, S.G.: ${\Sigma}^1_1$ and ${\Pi}^1_1$ transfinite induction. In: van
  Dalen, D., Lascar, D., Smiley, J. (eds.) Logic Colloquium '80. pp. 239--253.
  North Holland, Amsterdam (1982)

\bibitem{Simpson:2009:SubsystemsOfSecondOrderArithmetic}
Simpson, S.G.: Subsystems of {S}econd {O}rder {A}rithmetic. Cambridge
  University Press, New York (2009)

\end{thebibliography}

\end{document}